\documentclass[11pt]{article}
\usepackage[titletoc]{appendix}
\usepackage[labelsep=period]{caption}
\usepackage{bm}
\usepackage{fullpage}
\usepackage{epsfig}
\usepackage{graphics}
\usepackage{latexsym}
\usepackage{amsmath}
\usepackage{amsfonts}
\usepackage{amssymb}
\usepackage{mathrsfs}
\usepackage{pifont}
\usepackage{yhmath}

\usepackage{enumitem}
\usepackage{subcaption}
\usepackage{float}
\usepackage{mathtools}

\usepackage{bbm}
\usepackage{dsfont}
\usepackage{eufrak}
\usepackage{wasysym}
\usepackage{underscore}
\usepackage{epstopdf}
\usepackage{fontenc}
\usepackage{amsthm}
\newtheorem{theorem}{Theorem}[section]

\newtheorem{corollary}[theorem]{Corollary}

\newtheorem{lemma}[theorem]{Lemma}

\def\bfm#1{\mbox{\boldmath$#1$}}
\def\qed{\hfill \rule{4pt}{7pt}}
\usepackage{fullpage,graphicx}
\newcommand{\de}{\backslash}
\DeclareMathAlphabet{\mathpzc}{OT1}{pzc}{m}{it}
\title{\bf Tournament Ranking: Duality and Efficiency}
\author{\vspace{2mm}  Ge Song$^{a}$ \quad Mengxi Yang$^{b}$\thanks{Corresponding author. E-mail: yangmx221b@outlook.com.}
\quad Wenan Zang$^{c}$\thanks{Supported in part by the Research Grants Council of Hong Kong.}\\
$\stackrel{a}{}$ School of Public Policy and Management\\ University of Science and Technology of China\\Hefei 230026, China \smallskip\\
$\stackrel{b}{}$ School of Mathematical Sciences\\ University of Science and Technology of China\\Hefei 230026, China \smallskip\\
$\stackrel{c}{}$ Department of Mathematics\\University of Hong Kong \\ Hong Kong, China}

\begin{document}
\date{}
\maketitle

\begin{abstract}
The feedback arc set problem on tournaments arises in a rich variety of applications, and has been studied extensively in 
several research fields over the past six decades. It is well known that this problem is $NP$-hard and admits a polynomial-time 
approximation scheme (PTAS) in general. A tournament $T=(V, A)$ is called {\em cycle Mengerian} (CM) if, for every nonnegative 
integral weight function defined on $A$, the minimum total weight of a feedback arc set is equal to the maximum size of a cycle packing. 
In 2020 Chen et al. obtained a structural characterization of all CM tournaments; however, their proof is not algorithmic 
in nature. In this paper we present combinatorial polynomial-time algorithms for finding both minimum feedback arc 
sets and maximum cycle packings in arc-weighted CM tournaments.

\vskip 4mm


\noindent {\bf OR/MS subject classification.} Primary: Programming/graphs.


\end{abstract}

\newpage

\section{Introduction}

We address the classical problem of ranking a set of players on the basis of a set of pairwise comparisons arising from 
a sports tournament, with the objective of minimizing the total number of upsets, where an {\em upset} occurs if a higher 
ranked player was actually defeated by a lower ranked player. This problem can be rephrased as the so-called minimum 
feedback arc set problem on tournaments, and will be investigated in the more general weighted setting in this paper.  

Let $G=(V,A)$ be a digraph with a nonnegative integral weight $w(e)$ on each arc $e$.  A subset $F$ of arcs
is called a {\it feedback arc set} (FAS) of $G$ if $G\de F$ contains no cycles (directed). The {\em FAS problem}
is to find an FAS in $G$ with minimum total weight. One approach to this $NP$-hard problem is to formulate it as an integer 
program, consider its linear programming relaxation, and explore the integrality and duality properties satisfied by its 
constraint. Let ${\cal C}$ be the family of all cycles (directed) in $G$ and let $M$ be the $\cal C$-$A$ incidence matrix 
of $G$. Write
\begin{equation*}
\begin{aligned}
\tau_w(G) & := \min\{ w^T x: \ x \in {\mathbb Z}_+^A, \ M x \ge \bfm 1\},\\
\tau_w^*(G) & := \min\{w^T x: \ x \in {\mathbb R}_+^A, \ M x \ge\bfm 1\},\\ 
\nu_w(G) & := \max\{ y^T \bfm 1: \ y\in {\mathbb Z}_+^{\cal C}, \  y^TM\le w^T\},\\ 
\nu_w^*(G) & :=\max\{y^T \bfm 1: \ y\in {\mathbb R}_+^{\cal C}, \  y^TM\le w^T\}. 
\end{aligned}
\end{equation*}
\noindent (As usual, $\mathbb R_+$ and $\mathbb Z_+$ are the sets of nonnegative real numbers and nonnegative integers, respectively.) 
Combinatorially, $\tau_w(G)$ is the minimum total weight of an FAS in $G$. Each vector $y\in \mathbb Z_+^{\cal C}$ with $y^TM\le w^T$ can 
be interpreted as a collection $\cal Q$ of cycles (with repetition allowed) of $G$, such that each arc $e$ belongs to at most $w(e)$ members 
of $\cal Q$; such a collection is called a {\em cycle packing} of $G$. Thus $\nu_w(G)$ is the maximum size of a cycle packing of $G$, and 
naturally the corresponding maximization problem is called the {\em cycle packing problem}, which is the dual of the FAS problem. Observe that
\begin{equation}\label{eq:ineq}
\nu_w(G) \le \nu_w^*(G) = \tau_w^*(G) \le \tau_w(G),
\end{equation}
where the equality follows from the LP Duality Theorem. Seymour \cite{sey95} proved that in the unweighted case the supremum of $\tau_w(G)/\tau_w^*(G)$ is 
$O(\log \tau_w^*(G) \log\log \tau_w^*(G))$; his proof has been adapted by Even et al. \cite{ENSS} to get the best known approximation algorithm for the 
(weighted) FAS problem, with ratio $O(\log |V| \log\log |V|)$. We call $G$ {\it cycle ideal} (CI) if $\tau_w^*(G)=\tau_w(G)$ for all $w\in \mathbb Z_+^A$, 
and call $G$ {\it cycle Mengerian} (CM) if $\nu_w^*(G)=\nu_w(G)$ for all $w\in \mathbb Z_+^A$. By the Edmonds-Giles theorem \cite{eg77}, being 
CM is equivalent to saying that $\nu_w(G) = \tau_w(G)$ for all $w\in \mathbb Z_+^A$. Therefore, in view of (\ref{eq:ineq}), every CM digraph is CI.
We point out that characterizations of CI and CM digraphs can yield not only beautiful mathematical theorems but also a polynomial-time
solution of the FAS problem on such digraphs using the ellipsoid method, by a general theorem of Gr\"{o}tschel, Lov\'asz, and Schrijver \cite{GLS}.
Initiated in the early 1960s \cite{Y}, the study of CI and CM digraphs has inspired many min-max theorems in combinatorial optimization, such 
as Lucchesi and Younger \cite{LY}, Seymour \cite{sey77,sey96}, Guenin \cite{G,G2}, Geelen and Guenin \cite{GG},  Guenin and Thomas \cite{GT}, 
Chen et al. \cite{CDHZ,CDZZ1,CDZZ2,CDZZ3}, Ding, Feng, and Zang \cite{DFZ}, and Ding, Xu, and Zang \cite{DXZ,DZ02}. Despite tremendous research efforts, only 
some special classes of CI and CM digraphs \cite{ACM,BFM,CDZ,CDZZ1,CDZZ2,G,GT,LY,sey96} have been identified to date, and complete characterizations 
seem extremely hard to obtain.

A digraph $G$ is called a {\em tournament} if there is precisely one arc between any two vertices in $G$. The FAS problem remains $NP$-hard 
even when the input digraph $G$ is a tournament; see Alon \cite{Alon} and Charbit, Thomass\'e, and Yeo \cite{CTY}. As this special version also 
arises in a rich variety of applications, it has been studied extensively from the combinatorial \cite{EM,Fu,S,Y}, statistical \cite{Sl}, and algorithmic 
\cite{ACN,ACK,CFR,KS,vanW} points of view, and thus has produced a vast body of literature.  Mathieu and Schudy \cite{KS} devised a polynomial 
time approximation scheme (PTAS) for the FAS problem on tournaments.  Bessy et al. \cite{Bessy} showed that the problem of determining if a tournament 
has a cycle packing and a feedback arc set of the same size is $NP$-complete, and the problem of packing arc-disjoint cycles in tournaments is fixed-parameter 
tractable. Applegate, Cook, and McCormick \cite{ACM} and Barahona, Fonlupt, and Mahjoub \cite{BFM} independently proved that every tournament with five 
vertices is CM, thereby confirming a conjecture posed by both Barahona and Mahjoub \cite{BM} and J\"{u}nger \cite{J}. In \cite{CDZZ1,CDZZ2}, Chen et al.
proposed to call a tournament {\em M\"{o}bius-free} if it contains none of $K_{3, 3}$, $K_{3, 3}'$, $M_5$, and $M_5^*$ depicted in Figure 1 as a subgraph, and 
obtained the following characterization of all CI and CM tournaments.
\begin{figure}[htpb]
\vspace{-1mm}
\centerline{\includegraphics[width=10cm]{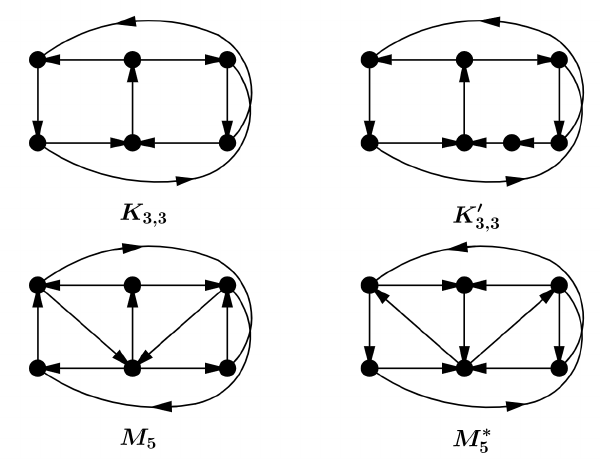}}
\vspace{-1mm}\caption{Forbidden Structures}\label{mladders}
\end{figure}
\begin{theorem} \label{Theorem1} {\rm (Chen et al. \cite{CDZZ1,CDZZ2})}
For a tournament $T$, the following statements are equivalent:
\begin{itemize}
\vspace{-1mm}
\item[(i)] $T$ is M\"{o}bius-free;
\vspace{-2mm}
\item[(ii)] $T$ is cycle ideal; and
\vspace{-2mm}
\item[(iii)] $T$ is cycle Mengerian.
\end{itemize}
\end{theorem}

Their proof, however, is not algorithmic in nature. The purpose of this paper is to present combinatorial polynomial-time algorithms for 
solving both the FAS problem and the cycle packing problem on M\"{o}bius-free tournaments exactly. Note that cycle packing is also a fundamental
problem in graph theory and algorithm design, with important applications in several fields. Hence it is also a subject of extensive 
research; see, for instance, Bessy et al. \cite{Bessy}, Caprara et al. \cite{CPR}, and Krivelevich et al. \cite{Kri}.  

\begin{theorem}\label{thm:main}
Let $T=(V,A)$ be a M\"{o}bius-free tournament with a nonnegative integral weight $w(e)$ on each arc $e$. Then  a minimum FAS in $(T, w)$ can be 
found in $O(n^9)$ time and a maximum cycle packing can be found in $O(n^7)$ time, where $n=|V|$.
\end{theorem}

The remainder of this paper is organized as follows. In Section 2, we describe the global structure of M\"{o}bius-free  
tournaments obtained by Chen et al. \cite{CDZZ3}, and reduce the cycle packing problem to that on a restricted class of 
M\"{o}bius-free tournaments. In Section 3, we propose a combinatorial polynomial-time algorithm for the FAS problem on this 
restricted class, using a dynamic programming approach. In Section 4, we devise a combinatorial polynomial-time algorithm for 
the cycle packing problem on this restricted class. In Section 5, we design combinatorial polynomial-time algorithms for
both the FAS problem and a cycle packing problem on all M\"{o}bius-free tournaments.  

\section{Preliminaries}

Our algorithms are built upon the following structure theorem about M\"{o}bius-free tournaments due to Chen et al. \cite{CDZZ3}. Recall that a 
digraph is called {\em strongly connected} or simply {\em strong} if each vertex is reachable from any other vertex. 

\begin{figure}[htpb]
\vspace{-1mm}
\centerline{\includegraphics[width=7cm]{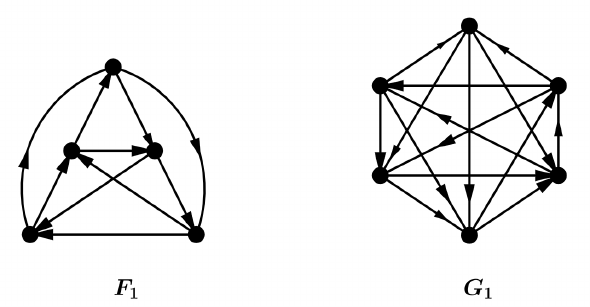}}
\vspace{-1mm}\caption{$F_1$ and $G_1$}\label{F1G1}
\end{figure}

\begin{theorem}\label{structure} {\rm (Chen et al. \cite{CDZZ3})}
Let $T=(V, A)$ be a strong tournament other than $F_1$ and $G_1$ (see Figure 2). Then $T$ is M\"obius-free if and only if it satisfies the description shown in Figure 3,
where $m\ge1$, undirected/dotted edges in the figure can be directed arbitrarily, and all other arcs (that are not drawn) are directed from ``left'' to ``right''. 
Furthermore, $v_1$ has an out-neighbor in the leftmost building block of $A_1$, and $v_m$ has an in-neighbor in the rightmost building block of $A_m$.
\end{theorem}

In this theorem, a {\em building block} of $A_i$ is a shaded stick containing $v_i$, labelled by $1,2, \ldots, n_i$, in Figure 3. What we mean by from ``left'' to ``right'' 
is from vertices on the left to those on the right. Moreover, each $A_i$ contains $v_i$ and each $B_i$ contains both $v_i$ and $v_{i+1}$. Following Chen et al. 
\cite{CDZZ3}, we call $A_1, A_2, \ldots, A_m$ {\em vertical blocks} of $T$, call $B_1, B_2, \ldots, B_{m-1}$ {\em horizontal blocks} of $T$, and call $v_1,v_2, 
\ldots, v_m$ the {\em join vertices} of $T$.  

\begin{figure}[htpb]
\centerline{\includegraphics[width=15cm]{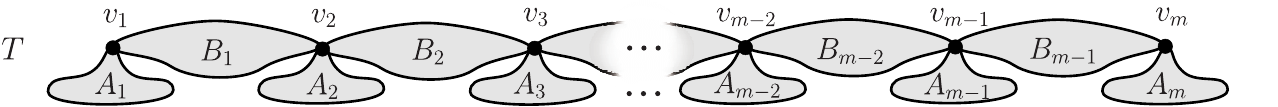}}
\end{figure}
\begin{figure}[htpb]
\centerline{\includegraphics[width=14cm]{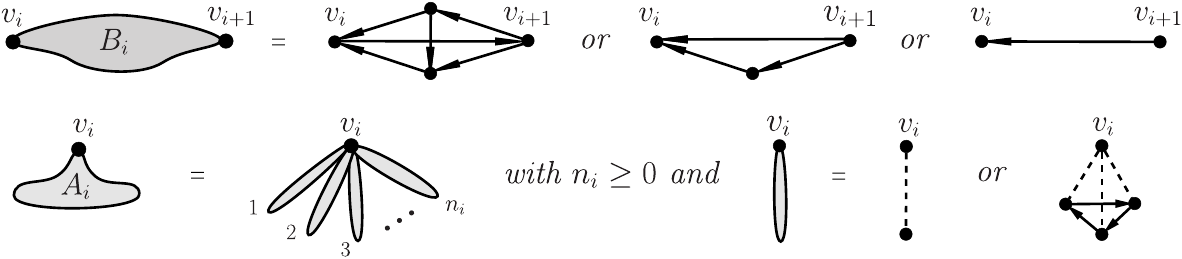}}
\caption{Global Structure}\label{figxj1}
\end{figure}

\vskip 1mm
The proof of Theorem \ref{structure} yields a polynomial-time algorithm for exhibiting the desired global structure. Since the authors did not  
estimate the complexity of this algorithm, we provide one below. Recall from \cite{CDZZ1} that a {\em dicut} in a digraph $G=(V,A)$ is a partition $(X,Y)$ of $V$ such 
that all arcs between $X$ and $Y$ are directed from $X$ to $Y$. A dicut $(X,Y)$ is called {\em trivial} if $|X|= 1$ or $|Y| = 1$. Furthermore, a digraph 
is called {\em weakly connected} if its underlying undirected graph is connected. A weakly connected digraph $G$ is called {\em internally strong} if 
every dicut of $G$ is trivial, and is called {\em internally 2-strong} (i2s) if $G$ is strong and $G\de v$ is internally strong for every vertex $v$.

\begin{corollary}\label{cor:exhibit global structure}
The global structure of $T$ described in Theorem \ref{structure} can be exhibited in $O(n^4)$ time, where $n=|V|$.
\end{corollary}
\vspace{-1mm}
{\bf Proof.} The structural description was carried out by induction on $n$ in Chen et al. \cite{CDZZ3}. The induction base is trivial.
To establish the induction step, first check if $T$ is i2s. If yes, then $T \in {\cal T}_0$, which consists of nine concrete tournaments, 
each with at most six vertices (see Theorem 1.2 in Chen et al. \cite{CDZZ1}). So the global structure of $T$ can be determined in 
constant time. Otherwise, apply Tarjan's algorithm \cite{tarjan} to find a dicut $(X,Y)$ in $T\de v$ for some vertex $v$ such that,
letting $T_1$ (resp. $T_2$) be the strong tournament obtained from $T$ by contracting $Y$ (resp. $X$) into a single vertex, 
$T_2 \in {\cal T}_1$, which comprises seven concrete tournaments, each with at most six vertices (see Theorem 1.3 
in Chen et al. \cite{CDZZ1}); this step requires $O(n^3)$ time. Thus the global structure of $T$ can be obtained by incorporating 
the small tournament $T_2$ into the global structure of $T_1$. Therefore the whole algorithm takes $O(n^4)$ time.  \qed\\

Throughout this paper, by a path or a cycle we always mean a directed one.  By a {\em triangle} we mean a directed cycle of
length three. Let $P$ be a path from $a$ to $b$. We say that $P$ is an $a$-$b$ {\em path}. Let $c$ and $d$ be two vertices on $P$ such that 
$a,c,d, b$ (not necessarily distinct) occur on $P$ in order as we traverse $P$ in its direction from $a$. We use $P[c,d]$ to denote the 
subpath of $P$ from $c$ and $d$, and set $P(c,d]=P[c,d]\de c$,  $P[c,d)=P[c,d]\de d$, and $P(c,d)=P[c,d]\de \{c,d\}$. Let $C$ be a cycle 
and let $u$ and $v$ be two vertices on $C$, we use $C[u,v]$ to denote the segment of $C$ from $u$ to $v$. For any digraph $H$, we use 
$V(H)$ and $A(H)$ to denote its vertex set and arc set, respectively, and set $|H|=|V(H)|$. For convenience, we also use $P$ (resp. $C$) 
to denote the arc set of a path $P$ (resp. a cycle $C$). An arc $e$ from $u$ to $v$ is denoted by $(u,v)$ or simply by $uv$; we call
$u$ the {\em tail} of $e$ and $v$ the {\em head} of $e$.

Let $\cal{G}$ denote the class of all tournaments $T$ as depicted in Figure 3, where $m\ge1$, undirected/dotted edges in the figure can be 
directed arbitrarily, and all other arcs (that are not drawn) are directed from ``left'' to ``right''. However, in contrast to Theorem \ref{structure}, 
$v_1$ might have no out-neighbor in the leftmost building block $A_{1,1}$ of $A_1$, and $v_m$ might have no in-neighbor in the rightmost 
building block $A_{m,n_m}$ of $A_m$. Note that such $T$ is not necessarily strong. Indeed, if $v_1$ has no out-neighbor 
in $A_{1,1}$, then all arcs with precisely one end in $A_{1,1}\de v_1$ are directed from $A_{1,1}\de v_1$. Since $A_{1,1}\de v_1$ is either a single
vertex or a triangle, we can obviously reduce the FAS and cycle packing problems on $T$ to those on $T'$, the subtournament 
of $T$ induced by all vertices outside $A_{1,1}\de v_1$, in this case. Similarly, reduction can be performed if $v_m$ has no in-neighbor in $A_{m,n_m}$.
We may repeat this ``tidy-up" process whenever possible.   

\vskip 1mm

\begin{lemma}\label{lem:obs on cycles}
Let $T$ be a tournament in ${\cal G}$. Then the following statements hold:
\begin{itemize}
\vspace{-2mm}
\item[(i)] All arcs between $\cup_{t=1}^{i-1} (A_t \cup B_t) \de v_i$ and $A_i\de v_i$ are directed to $A_i \de v_i$ for $2\le i \le m$, and all arcs 
between $A_i \de v_i$ and $(\cup_{t=i+1}^{m} A_t) \cup (\cup_{t=i}^{m-1} B_t) \de v_i$ are directed from $A_i \de v_i$ for $1\le i \le m-1$.

\vspace{-2mm}

\item[(ii)] If a cycle contains a vertex in $A_i$ but is not fully contained in $A_i \de v_i$, then  
it must pass through $v_i$ for $1\le i \le m$.
\vspace{-2mm}
\item[(iii)] Each $v_j$-$v_i$ path, with $1\le i<j\le m$, is fully contained in the horizontal blocks 
$B_{j-1}, B_{j-2},\\ \ldots, B_i$ and passes through the join vertices $v_j,v_{j-1}, \ldots,v_i$ in order.

\vspace{-2mm}
\item[(iv)] If $v_i$ (resp. $v_j$) is the join vertex on a cycle $C$ with the smallest (resp. largest) subscript, then
$C[v_i, v_j]$ contains no join vertex other than $v_i, v_j$ and contains at most one non-join vertex from each horizontal 
block $B_k$ for $i \le k \le j-1$. (Note that $C[v_i, v_j]$ if $i<j$ and $C$ otherwise may contain two non-join vertices 
from $B_{i-1}$ if $|B_{i-1}|=4$ or from $B_j$ if $|B_j|=4$.)  
\end{itemize}
\end{lemma}
\vspace{-1mm}

{\bf Proof.} Since arcs of $T$ not shown in Figure 3 are all directed from ``left" to ``right", (i) holds trivially, which 
implies both (ii) and (iii) directly.  

(iv) From (iii) we see that $C[v_i, v_j]$ contains no join vertex other than $v_i$ and $v_j$. If it contains two non-join vertices from some 
horizontal block $B_k$, with $i \le k \le j-1$. Then $|B_k|=4$ and thus $C[v_i, v_j]$ and $C[v_j, v_i]$ would have a common vertex in $B_k$ 
(see the the structure of $B_k$ in Figure 3), a contradiction. \qed\\

Let $\cal{G}^*$ denote the class of all tournaments in $\cal{G}$, in which $A_i\de v_i$ is acyclic for all vertical blocks $A_i$ (that is, 
each building block of each $A_i$ has precisely two vertices). This restricted class will play an important role in our design of algorithms. 
Let us show that the cycle packing problem on ${\cal G}$ can be reduced to that on ${\cal G}^*$.

\begin{lemma}\label{lem:packtriangle}
Let $T=(V,A)$ be a tournament in ${\cal G}$ with a nonnegative integral weight $w(e)$ on each arc $e$, let $\mathcal{R}$ consist of all triangles 
contained in $A_i \de v_i$ for all $i$, and let $\delta(S):=\min\{w(e): e\in S\}$ for $S\in \mathcal{R}$. Then there exists a maximum cycle packing $y$ 
in $(T, w)$, such that $y(S)=\delta(S)$ for all $S\in \mathcal{R}$.
\end{lemma}
\vspace{-1mm}
{\bf Proof.} Let $y$ be a maximum cycle packing in $T$ such that 

\vskip 1mm
(1) $\sum_{S\in\mathcal{R}} y(S)$ is maximized. 

\vskip 1mm
We aim to prove that $y$ is as desired. Assume the contrary: $y(S)<\delta(S)$ for some triangle $S\in\mathcal{R}$. Then $w(e) \ge \delta(S)>0$
for each arc $e \in S$. Let ${\cal C}$ be the family of all cycles in $T$. Define $q(e):=\sum_{e\in C\in\cal{C}} y(C)$ and $w'(e):=w(e)-q(e)$ 
for each arc $e$ of $T$. Then both $q(e)$ and $w'(e)$ are nonnegative integers. Let $\mathcal{C}'$ denote the family of cycles in $T$ with $y(C)>0$ 
and let $A_i$ be the vertical block containing $S$. Let $a$, $b$ and $c$ be the vertices of $S$, with arcs $ab$, $bc$ and $ca$. From the maximality 
assumption on $y$, we deduce that at least one of $w'(ab)$, $w'(bc)$ and $w'(ca)$ is $0$. Depending on their values, we proceed by considering three cases.

{\bf Case 1.} Exactly one of $w'(ab)$, $w'(bc)$ and $w'(ca)$ is $0$, say $w'(ab)=0$. In this case, $q(ab)=w(ab)> 0$. So there exists a cycle $C'\in\mathcal{C}'\de\{S\}$, 
with $ab \in C'$. Since the triangles in $\mathcal{R}$ are pairwise vertex-disjoint, $C'\notin\mathcal{R}$. Define $y':\mathcal{C}\rightarrow\mathbb{Z}_+$ by
\[
y'(C)=
\begin{dcases}
y(C)-1 &\text{ if }C=C',\\
y(C)+1 &\text{ if }C=S,\\
y(C) &\text{ otherwise}.
\end{dcases}
\]
Clearly, $\sum_{C\in\mathcal{C}} y'(C)=\sum_{C\in\mathcal{C}}y(C)$.  Observe that 

\vskip 1mm
$\bullet$ $\sum_{e \in C\in\mathcal{C}} y'(C)\le q(e)+1= w(e)-w'(e)+1\le w(e)$ if $e \in \{bc, ca\}$, and 

\vskip 1mm
$\bullet$ $\sum_{e\in C\in\mathcal{C}}y'(C)\le q(e)\le w(e)$ otherwise. 

\vskip 1mm
\noindent Hence $y'$ is a maximum cycle packing in $T$ with $\sum_{S\in\mathcal{R}} y'(S)> \sum_{S\in\mathcal{R}} y(S)$, contradicting (1). 

\vskip 1mm
{\bf Case 2.} Exactly two of $w'(ab)$, $w'(bc)$ and $w'(ca)$ are 0, say $w'(ab)=w'(bc)=0$. In this case, $q(ab)=w(ab)>0$ and $q(bc)=w(bc)>0$. 
So there exist cycles $C_1$ and $C_2$ in $\mathcal{C}'$, with $ab\in C_1$ and $bc\in C_2$. By Lemma \ref{lem:obs on cycles}(ii), both $C_1$ 
and $C_2$ pass through the join vertex $v_i$. 

If one of $C_1$ and $C_2$, say the former, contains both $ab$ and $bc$, define $y':\mathcal{C}\rightarrow\mathbb{Z}_+$ by
\[
y'(C)=
\begin{dcases}
y(C)-1 &\text{ if }C=C_1,\\
y(C)+1 &\text{ if }C=S,\\
y(C) &\text{ otherwise}.
\end{dcases}
\]
It is routine to check that $y'$ is a maximum cycle packing with $\sum_{S\in\mathcal{R}}y'(S)>\sum_{S\in\mathcal{R}}y(S)$, contradicting (1).

So we may assume that $bc \notin C_1$ and $ab \notin C_2$. From the structural description of $T$ (see Theorem \ref{structure}), we see 
that both $C_1[b,v_i]$ and $C_2[v_i,b]$ are arc-disjoint from $S$. Let $C_3$ be a cycle in $C_1[b,v_i]\cup C_2[v_i,b]$ passing through $v_i$. 
Define $y':\mathcal{C}\rightarrow\mathbb{Z}_+$ by
\[
y'(C)=
\begin{dcases}
y(C)-1 &\text{ if }C=C_1\text{ or }C_2,\\
y(C)+1 &\text{ if }C=S\text{ or }C_3,\\
y(C) &\text{ otherwise}.
\end{dcases}
\]
Then $y'$ is a maximum cycle packing in $T$ with $\sum_{S\in\mathcal{R}}y'(S)>\sum_{S\in\mathcal{R}}y(S)$, contradicting (1) again.

\vskip 1mm
{\bf Case 3.} $w'(ab)=w'(bc)=w'(ca)=0$. In this case, there exist cycles $C_1$, $C_2$ and $C_3$ in $\mathcal{C}'$ passing through $v_i$, with $ab\in C_1$, 
$bc\in C_2$ and $ca\in C_3$. 

We first assume that one of $C_1$, $C_2$ and $C_3$ contains two arcs of $S$, say $\{ab, bc\} \subseteq C_1$. Then both $C_1[v_i,a]$ and $C_1[c,v_i]$ are 
arc-disjoint from $S$. Furthermore, at least one of the paths $C_3[a,v_i]$ and $C_3[v_i,c]$ is arc-disjoint from $S$. So either $C_1[v_i,a]\cup C_3[a,v_i]$ 
or $C_1[c,v_i]\cup C_3[v_i,c]$ contains a cycle $C_4$, passing through $v_i$, that is arc-disjoint from $S$. Define $y':\mathcal{C}\rightarrow\mathbb{Z}_+$ by
\[
y'(C)=
\begin{dcases}
y(C)-1 &\text{ if }C=C_1\text{ or }C_3,\\
y(C)+1 &\text{ if }C=S\text{ or }C_4,\\
y(C) &\text{ otherwise}.
\end{dcases}
\]
Then $y'$ is a maximum cycle packing in $T$ with $\sum_{S\in\mathcal{R}}y'(S)>\sum_{S\in\mathcal{R}}y(S)$, contradicting (1).

So we may assume that $C_j$ contains exactly one arc in $S$ for $j=1,2,3$. Then $C_1[v_i,a]\cup C_3[a,v_i]$ contains a cycle $C'_1$, 
$C_1[b,v_i]\cup C_2[v_i,b]$ contains a cycle $C'_2$, and $C_2[c,v_i]\cup C_3[v_i,c]$ contains a cycle $C'_3$, all passing through $v_i$. 
Note that these three cycles are arc-disjoint from $S$. Define $y':\mathcal{C}\rightarrow\mathbb{Z}_+$ by
\[
y'(C)=
\begin{dcases}
y(C)-1 &\text{ if }C=C_1,C_2\text{ or }C_3,\\
y(C)+1 &\text{ if }C=S,C'_1,C'_2\text{ or }C'_3,\\
y(C) &\text{ otherwise}.
\end{dcases}
\]
Then $y'$ is a cycle packing in $T$ having size larger than that of $y$; this contradiction completes the proof. \qed\\

Let $T=(V,A), w(e), \mathcal{R}$ and $\delta(S)$ be as specified in Lemma \ref{lem:packtriangle}, let $e_S$ be an arc of $S$ with $w(e_S)=\delta(S)$
for $S\in {\cal R}$, and let $T'=(V,A')$ be the tournament arising from $T$ by reversing the direction of $e_S$ for each $S \in {\cal R}$; let 
$\bar{e}_S$ denote the reverse of $e_S$. Note that $T'\in {\cal G}^*$. Once again, we use ${\cal C}$ to denote the family of all cycles in $T$, 
and use ${\cal C}'$ to denote the family of all cycles in $T'$. Consider the weight function $w': A'\rightarrow\mathbb{Z}_+$: 
\[
w'(e)=
\begin{dcases}
w(e)-\delta(S) &\text{ if } e\in A'\cap S \text{ for some } S\in\mathcal{R},\\
0 &\text{ if } e= \bar{e}_S \text{ for some }S\in\mathcal{R},\\
w(e) &\text{ otherwise}.
\end{dcases}
\]

\begin{lemma}\label{thm:notriangle packing}
Let $y'$ be a maximum cycle packing in $(T', w')$. Define $y:\mathcal{C}\rightarrow\mathbb{Z}_+$ by
\[
y(C)=
\begin{dcases}
\delta(C) & \text{ if } C\in\mathcal{R},\\
y'(C) &\text{ if } C\in\mathcal{C}\cap\mathcal{C}',\\
0 &\text{ otherwise}.
\end{dcases}
\]
Then $y$ is a maximum cycle packing in $(T, w)$.
\end{lemma}

{\bf Proof.} For each $S\in {\cal R}$, we have $w'(e_S)=0$, so $y'(C)=0$ for all $C\in {\cal C}'$ containing $e_S$. Since 
$\sum_{e\in C\in\mathcal{C}'} y'(C)\le w'(e)$ for any $e\in A'$, it is then routine to check that $\sum_{e\in C\in\mathcal{C}} 
y(C)\le w(e)$ for any $e\in A$. Hence $y$ is a cycle packing in $(T, w)$.

To see maximality enjoyed by $y$, let $\bar{y}$ be a maximum cycle packing in $(T, w)$ 
satisfying $\bar y(S)=\delta(S)$ for all $S\in\mathcal{R}$; such $\bar{y}$ is available by Lemma \ref{lem:packtriangle}.
Note that $\bar y(C)=0$ for $C\in\mathcal{C}\de(\mathcal{C}'\cup\mathcal{R})$, since every cycle in $\mathcal{C}\de\mathcal{C}'$ 
contains an arc $e_S$ for some $S\in\mathcal{R}$, which is already saturated by $\bar y(S)$. Thus
\[\sum_{C\in\mathcal{C}}\bar y(C)=\sum_{C\in\mathcal{C}\cap\mathcal{C}'}\bar y(C)+\sum_{C\in\mathcal{R}}\bar y(C)
\le \sum_{C\in \mathcal{C}'} y'(C)+\sum_{C\in\mathcal{R}} \delta(C)
=\sum_{C\in\mathcal{C}\cap\mathcal{C}'} y'(C)+\sum_{C\in\mathcal{R}} \delta(C)=\sum_{C\in\mathcal{C}} y(C).\]
Therefore $y$ is also a maximum cycle packing in $(T, w)$.  \qed

\section{Feedback Arc Sets}

The purpose of this section is to present a combinatorial polynomial-time algorithm for finding minimum FAS's in arc-weighted
tournaments in ${\cal G}^*$ using a dynamic programming approach and network flow techniques. 

\begin{theorem}\label{FAS}
Let $T=(V,A)$ be a tournament in ${\cal G}^*$ with a nonnegative integral weight $w(e)$ on each arc $e$. Then a minimum FAS in $(T, w)$
can be found in $O(n^5)$ time, where $n=|V|$.
\end{theorem}

Recall that $T$ is as depicted in Figure 3, in which $A_i \de v_i$ is acyclic for all vertical blocks $A_i$ (that is, each building 
block of each $A_i$ has precisely two vertices). For each horizontal block $B_i$, we reserve the symbol $a_i$ for the vertex in
$B_i \backslash \{v_i, v_{i+1}\}$ if $|B_i|=3$, and reserve the symbols $a_i$ and $b_i$ for the vertices in $B_i \backslash 
\{v_i, v_{i+1}\}$ with $a_ib_i\in A(B_i)$ if $|B_i|=4$. An arc subset $\pi$ of $B_i$ is called a $v_{i+1}$-${v_i}$ {\em cut} of $B_i$
if $B_i \de \pi$ contains no $v_{i+1}$-${v_i}$ path.  An inclusionwise minimal $v_{i+1}$-${v_i}$ cut $\sigma$ of $B_i$, for any $1\le i \le m-1$, 
is called a {\em horizontal cut} of $T$.  Let $B_{\sigma \ell}$ be the subgraph of $B_i$ induced by all vertices that are reachable 
to $v_{i}$ in $B\de \sigma$ and let $B_{\sigma r}:=B_i\de V(B_{\sigma \ell})$. Let $T_{\sigma \ell}$ denote the subtournament of $T$ induced by all vertices in
$(\cup_{t=1}^i A_t) \cup (\cup_{t=1}^{i-1} B_t) \cup B_{\sigma \ell}$ and let $T_{\sigma r}$ denote the subtournament of $T$ induced by all vertices in
$(\cup_{t=i+1}^m A_t) \cup (\cup_{t=i+1}^{m-1} B_t) \cup B_{\sigma r}$. Intuitively, $T_{\sigma \ell}$ (resp. $T_{\sigma r}$) is the subtournament of
$T$ on the left (resp. right) of the horizontal cut $\sigma$ in Figure 3. Note that in $T\de \sigma$ all arcs between $T_{\sigma \ell}$ and
$T_{\sigma r}$ are directed to $T_{\sigma r}$, as arcs not shown in Figure 3 are directed from ``left" to ``right". 

Let us exhibit some properties satisfied by inclusionwise minimal FAS's in $T$.

\begin{lemma}\label{lem:FAS}
Let $F$ be a minimum FAS in $T$. If $F$ contains a horizontal cut $\sigma$, then $A(T_{\sigma \ell})\cap F$ (resp. $A(T_{\sigma r})\cap F$) 
is a minimum FAS in $T_{\sigma \ell}$ (resp. $T_{\sigma r}$).
\end{lemma}
\vspace{-1mm}
{\bf Proof.} The statement is straightforward. \qed

\vspace{1mm}

\begin{lemma}\label{lem:arc in FAS}
Let $F$ be an inclusionwise minimal FAS in $T$. Then, for any $uv \in F$, there exists a $v$-$u$ path $P$ in $T$ with $P\cap F=\emptyset$.
Furthermore, every $u$-$v$ path $Q$ in $T$ satisfies $Q\cap F\neq\emptyset$.
\end{lemma}
\vspace{-1mm}

{\bf Proof.} Assume the contrary: Every $v$-$u$ path contains an arc in $F$. Then $F\de\{uv\}$ would remain to be an FAS in $T$, contradicting the minimality 
assumption on $F$. Hence $P\cap F=\emptyset$ for some $v$-$u$ path $P$. If $Q\cap F=\emptyset$ for some $u$-$v$ path $Q$, then every cycle contained in 
$P\cup Q$ would be disjoint from $F$, a contradiction. \qed\\

Throughout we reserve the symbol $\mathcal{Q}$ for the family of all $v_i$-$v_j$ paths $Q$, with $1\le i\le j\le m$, that contain no join vertex other 
than $v_i$ and $v_j$, here $Q$ is a cycle passing through $v_i$ when $i=j$.

\begin{lemma}\label{lem:FAS with no horizontal cut}
Let $F$ be an inclusionwise minimal FAS in $T$ containing no horizontal cut. Then the following statements hold:
\begin{itemize}
\vspace{-2mm}
\item[(i)] For any $1\le i<j\le m$, there exists a $v_j$-$v_i$ path $P$ with $P\cap F=\emptyset$.
\vspace{-2mm}
\item[(ii)] For any $Q\in\mathcal{Q}$, we have $Q\cap F\neq\emptyset$.
\vspace{-2mm}
\item[(iii)] If $a_ib_i\in F$ for some $i$ with $|B_i|=4$, then exactly one of $a_iv_i$ and $v_{i+1}b_i$ belongs to $F$.
\vspace{-2mm}
\item[(iv)] For any $v_i$-$v_j$ path $Q\in\mathcal{Q}$ with $i<j$, $Q\cap F$ contains at least one arc outside 
$\{a_kb_k: 1\le k\le m-1\ and\ |B_k|=4\}$.
\vspace{-2mm}
\item[(v)] If $|B_i|= 2$, then $v_{i+1}v_i\notin F$. If $|B_i|= 3$, then $v_{i+1}v_i\notin F$ and at most one of $a_iv_i$ and $v_{i+1}a_i$ 
belongs to $F$. Moreover, if $|B_i|=4$, then $v_iv_{i+1}\in F$ and exactly one of the following three cases occurs: 

$\bullet$ $(A(B_i)\de\{v_iv_{i+1}\})\cap F=\emptyset$; 

$\bullet$ $\{a_iv_i,a_ib_i,v_{i+1}b_i\}\cap F=\emptyset$ and exactly one of $b_iv_i$ and $v_{i+1}a_i$ belongs to $F$; and 

$\bullet$ $a_ib_i\in F$, $\{b_iv_i,\ v_{i+1}a_i\}\cap F=\emptyset$, and exactly one of $a_iv_i$ and $v_{i+1}b_i$ belongs to $F$.
    
\end{itemize}
\end{lemma}

{\bf Proof.} (i) Since $F$ contains no horizontal cut, there exists a $v_{k+1}$-$v_k$ path $P_k$ in $B_k$ that is disjoint from $F$ 
for $i \le k\le j-1$. Thus $\cup_{i \le k\le j-1} P_k$ is a $v_j$-$v_i$ path $P$ with $P\cap F=\emptyset$.

(ii) Assume the contrary: there exists a $v_i$-$v_j$ path $Q\in\mathcal{Q}$ with $Q\cap F=\emptyset$. If $i=j$, then $Q$ is a cycle with 
$Q\cap F=\emptyset$, contradicting the hypothesis that $F$ is an FAS. So we assume that $i<j$. By (i), there exists a $v_j$-$v_i$ path $P$ 
with $P\cap F=\emptyset$. Thus each cycle in $P\cup Q$ is disjoint from $F$, a contradiction again.
        
(iii) Let $P$ be a $b_i$-$a_i$ path with $P\cap F=\emptyset$; such $P$ is available by Lemma \ref{lem:arc in FAS}. From the global structure 
of $T$, we see that either $v_i$ or $v_{i+1}$ belongs to $P$. If $P$ passes through $v_i$, then $a_iv_i\in F$, for otherwise, 
$P[v_i,a_i] \cup \{a_iv_i\}$ would be a cycle disjoint from $F$, a contradiction. Similarly, if $P$ passes through $v_{i+1}$, 
then $v_{i+1}b_i\in F$. Thus $\{a_iv_i, v_{i+1}b_i\}\cap F\neq\emptyset$. Since $\{a_iv_i,a_ib_i,v_{i+1}b_i\}$ is a horizontal cut of $T$, 
exactly one of $a_iv_i$ and $v_{i+1}b_i$ belongs to $F$.

(iv) Assume the contrary: $Q \cap F \subseteq \{a_kb_k: 1\le k\le m-1\ {\rm and} \ |B_k|=4\}$.  We first consider the case when $|Q\cap F|=1$, say 
$Q\cap F=\{a_kb_k\}$ for some $i\le k\le j-1$. By (iii), either $a_kv_k\notin F$ or $v_{k+1}b_k\notin F$. 
If $a_kv_k\notin F$, in view of (i), letting $P$ be a $v_k$-$v_i$ path with $P\cap F=\emptyset$ (when $k=i$, $P$ is the singleton $v_i$), then 
$P\cup Q[v_i,a_k]\cup\{a_kv_k\}$ would contain a cycle disjoint with $F$, a contradiction. Similarly, we can reach a contradiction if $v_{k+1}b_k\notin F$.

So we assume that $|Q\cap F|=p>1$. Let $Q\cap F=\{a_{k_1}b_{k_1}, a_{k_2}b_{k_2},\ldots, a_{k_p}b_{k_p}\}$, where $i\le k_1<k_2<\ldots<k_p\le j-1$. 
Note that $Q$ passes $a_{k_1}b_{k_1}, a_{k_2}b_{k_2},\ldots, a_{k_p}b_{k_p}$ in order, given the global structure of $T$. 

By the same argument as used for the case when $|Q\cap F|=1$, we obtain $a_{k_1}v_{k_1}\in F$, which in turn implies from (iii) that 
$v_{k_1+1}b_{k_1}\notin F$. By (i), there exists a $v_{k_2}$-$v_{k_1+1}$ path $P$ with $P\cap F=\emptyset$ (when $k_2=k_1+1$, $P$ is 
a singleton). Observe that $a_{k_2}v_{k_2}\in F$, for otherwise, $\{v_{k_1+1}b_{k_1}\}\cup Q[b_{k_1},a_{k_2}]\cup\{a_{k_2}v_{k_2}\}\cup P$ would 
be a cycle disjoint with $F$. So $v_{k_2+1}b_{k_2}\notin F$ by (iii). Repeated application of the same argument yields  
$a_{k_p}v_{k_p}\in F$ and $v_{k_p+1}b_{k_p}\notin F$. By (i), there exists a $v_j$-$v_{k_p+1}$ path $P'$ with $P'\cap F=\emptyset$. Thus 
$\{v_{k_p+1}b_{k_p}\}\cup Q[b_{k_p},v_j]\cup P'$ would contain a cycle disjoint from $F$, this contradiction justifies (iv).

(v) When $|B_i|=2$, arc $v_{i+1}v_i$ forms a horizontal cut of $T$, so $v_{i+1}v_i \notin F$. 

When $|B_i|=3$, we have $v_{i+1}v_i\notin F$, for otherwise, at least one arc on the path $v_{i+1}a_iv_i$ is contained in $F$ by Lemma \ref{lem:arc in FAS}.
Thus $F$ contains a horizontal cut, $\{v_{i+1}v_i, a_iv_i\}$ or $\{v_{i+1}v_i, v_{i+1}a_i\}$, of $T$, contradicting the  
hypothesis on $F$. If both $a_iv_i$ and $v_{i+1}a_i$ belong to $F$, then Lemma \ref{lem:arc in FAS} guarantees the existence of a $v_i$-$a_i$ path 
$P_1$ and a $a_i$-$v_{i+1}$ path $P_2$ with $P_1\cap F=\emptyset$ and $P_2\cap F=\emptyset$. Thus every cycle contained in $P_1\cup 
P_2 \cup \{v_{i+1}v_i\}$ would be disjoint from $F$, a contradiction.

When $|B_i|=4$, we have $v_iv_{i+1}\in F$ by (ii). In what follows, we assume that the first case in the bulleted list does not occur; that is, $(A(B_i)\de\{v_iv_{i+1}\})\cap 
F\neq \emptyset$, and aim to show that one of the remaining two cases arises. 

We first assume that $a_ib_i\notin F$. If $a_iv_i\in F$, then $b_iv_i\in F$ by Lemma \ref{lem:arc in FAS} and using the path $a_ib_iv_i$. Thus
$\{a_iv_i, b_iv_i\}$ is a horizontal cut of $T$ contained in $F$, a contradiction.  So $a_iv_i\notin F$,
Similarly, we can show that $v_{i+1}b_i\notin F$. Thus exactly one of $b_iv_i$ and $v_{i+1}a_i$ belongs to $F$, since $(A(B_i)\de\{v_iv_{i+1}\})\cap F
\neq \emptyset$ and $F$ contains no horizontal cut. Therefore, the second case occurs.

We next assume that $a_ib_i\in F$. Now exactly one of $a_iv_i$ and $v_{i+1}b_i$ belongs to $F$ by (iii). Symmetry allows us to assume that 
$a_iv_i\in F$ and $v_{i+1}b_i\notin F$, which in turn implies $b_iv_i\notin F$, since $F$ contains no horizontal cut.   
If $v_{i+1}a_i\in F$, then there exists an $a_i$-$v_{i+1}$ path $P$ with $P\cap F=\emptyset$ by Lemma \ref{lem:arc in FAS}. Thus 
$P\cup \{v_{i+1}b_i,b_iv_i\}$ would be an $a_i$-$v_i$ path that is disjoint from $F$, contradicting Lemma \ref{lem:arc in FAS} (for $a_iv_i\in F$). 
Therefore, the third case occurs. \qed\\

Let $w$ be the weight function defined on the arc set $A$ of $T$ as specified in Theorem \ref{FAS}. To derive further properties
satisfied by FAS's of $T$ that contain no horizontal cuts, we construct a flow network $N$ from $(T, w)$ as follows:

\begin{itemize}

\vspace{-1mm}
\item Delete $v_{i+1}v_i$ for all $i$ with $|B_i|\le 3$, and delete $a_ib_i$ for all $i$ with $|B_i|=4$; 
\vspace{-1mm}
\item Split $v_i$ for all $1\le i \le m$ into two new vertices $s_i$ and $t_i$, replace each arc $e$ of $T$ with tail (resp. head) $v_i$ by a new arc $e'$ 
with tail $s_i$ (resp. with head $t_i$). Define the capacity $c(e')=w(e)$; and 
\vspace{-1mm}
\item Split each $u_i \in \{a_i, b_i\}$ for all $1\le i \le m-1$ into two new vertices $u_{i1}$ and $u_{i2}$, replace each arc $e$ of $T$ between $u_i$ and 
$(\cup_{t=1}^i A_t)\cup(\cup_{t=1}^{i-1} B_t)$ by a new arc $e'$ that enters (resp. leaves) $u_{i1}$ if $e$ enters (resp. leaves) 
$u_i$, and replace each arc $e$ of $T$ between $u_i$ and $(\cup_{t=i+1}^m A_t) \cup (\cup_{t=i+1}^{m-1} B_t)$ by a new arc $e'$ that enters (resp. leaves) 
$u_{i2}$ if $e$ enters (resp. leaves) $u_i$. Define the capacity $c(e')=w(e)$. Then add an arc $u_{i1}u_{i2}$ and define its capacity $c(u_{i1}u_{i2})=\infty$. 
Finally, add two arcs $a_{i1}b_{i1}$ and $a_{i2}b_{i2}$ and define their capacities $c(a_{i1}b_{i1})=c(a_{i2}b_{i2})=w(a_ib_i)$ for all $i$ with $|B_i|=4$.
\vspace{-1mm}
\end{itemize}

This construction is illustrated in Figure 4. Note that $u_i=a_i$ if $|B_i|=3$ (as $b_i$ is undefined now) and that the above splitting operation applies to 
every vertex in horizontal blocks of $T$ (but no other vertices). As a consequence, for example, the arc $v_1a_5$ in $T$ will become the arc $s_1a_{51}$ 
in $N$, and the arc $a_2u$ in $T$, with $u \in V(A_6)\de v_6$, will become the arc $a_{22}u$ in $N$. Clearly, $N$ can be constructed in $O(n^2)$ time, where 
$n=|V|$.     

Let $V'$ be the vertex set of $N$, let $A'$ consist of all $e'$, all $a_{i1}b_{i1}$ and $a_{i2}b_{i2}$, and let $A''$ consist of all $u_{i1}u_{i2}$.
Then $N=(V', A' \cup A'')$. For convenience, we define a surjection $\eta: A' \rightarrow A$ by $\eta(e')=e$ and $\eta(a_{i1}b_{i1})=\eta(a_{i2}b_{i2})=a_ib_i$.
Let $\Lambda:=\{s_1,s_2, \ldots, s_m\}$ and $\Pi:=\{t_1, t_2, \ldots, t_m\}$. We view $\Lambda$ (resp. $\Pi$) as the source (resp. sink) set of $N$, and call the 
quadruple $(N, \Lambda, \Pi, c)$ (or simply the pair $(N, c)$) the {\em corresponding network} of $(T,w)$.  As usual, a $(\Lambda, \Pi)$-{\em cut} in $N$ is a 
subset $K$ of arcs that contains at least one arc from each path connecting $\Lambda$ and $\Pi$, whose capacity is defined to be $c(K):=\sum_{a\in K} c(a)$. A 
$(\Lambda, \Pi)$-{\em flow} in $N$ is a function $f: A' \cup A'' \rightarrow \mathbb{R}_+$ such that 

$\bullet$ $0\le f(a) \le c(a)$ for each arc $a \in A$;

$\bullet$ ${\rm div}_f(v)=0$ if $v \in V'\de (\Lambda \cup \Pi)$, ${\rm div}_f(v)\ge 0$ if $v \in \Lambda$, and ${\rm div}_f(v)\le 0$ if $v \in \Pi$, 

\noindent where ${\rm div}_f(v) = f(\delta^+(v))- f(\delta^-(v))$, and $\delta^+(v)$ (resp. $\delta^-(v)$) is the set of all arcs leaving (resp. entering) $v$
in $N$. For each $U \subseteq V'$, set ${\rm div}_f(U):=\sum_{v\in U} {\rm div}_f(v)$. It is easy to see that ${\rm div}_f(\Lambda)=-{\rm div}_f(\Pi)$; this 
quantity is called the {\em value} of $f$, denoted by ${\rm val}(f)$. When $\Lambda=\{s_1\}$ and $\Pi=\{t_1\}$, a $(\Lambda, \Pi)$-cut is often called an $s_1$-$t_1$ 
{\em cut}, and a $(\Lambda, \Pi)$-flow is often called an $s_1$-$t_1$ flow. 

\vskip 4mm
\begin{figure}[htpb]
\centerline{\includegraphics[width=10cm]{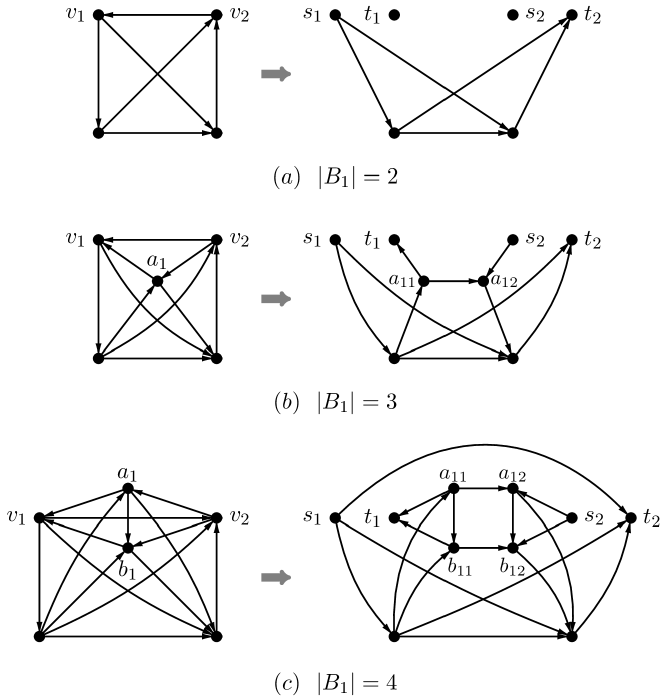}}
\caption{Network Construction}
\end{figure}

The above definition of $(\Lambda, \Pi)$-flow is given in the arc-vertex form. An equivalent definition is formulated in the following path packing form: 
Let ${\cal P}$ be the set of all simple paths in $N$ from $\Lambda$ to $\Pi$. An $(\Lambda, \Pi)$-flow is an assignment $g:\, {\cal P}\rightarrow \mathbb R_+$ 
such that $\sum_{e \in P \in {\cal P}}\, g(P) \le  c(e)$ for all arcs $e$, whose value is $\sum_{P \in {\cal P}}\, g(P)$. We can view $g(P)$ as the multiplicity
of $P$, and shall use both forms in our proofs.

By the max-flow min-cut theorem, the maximum value of a $(\Lambda, \Pi)$-flow is equal to the minimum capacity of a $(\Lambda, \Pi)$-cut in  
$N$. Let us exhibit some further properties satisfied by $N$. 

\begin{lemma}\label{lem:paths in tildeT}
There is no $s_j$-$t_i$ path in $N$ for any $1\le i<j\le m$.
\end{lemma}
\vspace{-1mm}

{\bf Proof.} Let $R_j$ be the set of all arcs between $(\cup_{t=1}^{j-1} A_t) \cup (\cup_{t=1}^{j-2} B_t)$ and $(\cup_{t=j}^m A_t) \cup (\cup_{t=j}^{m-1} B_t)$ 
in the tournament $T$. By Theorem \ref{structure}, all arcs in $R_j$ are directed from the former to the latter, except $v_jv_{j-1}$ when $|B_{j-1}|\le 3$,
which has been deleted in the construction of $N$. Let $R_j':=R_j\de \{v_jv_{j-1}\}$ if $|B_{j-1}|\le 3$ and $R_j':=R_j$ otherwise, and let 
$Z=\{e': e\in R_j'\} \cup \{u_{(j-1)1}u_{(j-1)2}: u_{j-1}=a_{j-1}\,\, or \,\, b_{j-1}\}$.  Then $Z$ forms a dicut $(X,Y)$ in $N$ with $t_i \in X$ and $s_j \in Y$.  
Hence $N$ contains no $s_j$-$t_i$ path. \qed\\

The following lemma establishes a correspondence between some paths in $N$ and some in $T$.

\begin{lemma}\label{lem:btw paths}
For any $s_i$-$t_j$ path $P$ in $N$ with $i\le j$, the arcs in $\{\eta(a): a\in P \cap A'\}$ form a $v_i$-$v_j$ path $Q$ in $\mathcal{Q}$. Conversely,
for any $v_i$-$v_j$ path $Q$ in $\mathcal{Q}$ with $i\le j$, there exists an $s_i$-$t_j$ path $P$ in $N$, such that the arcs in $\{\eta(a): a\in P \cap A' \}$ form $Q$.
\end{lemma}
\vspace{-1mm}

{\bf Proof.} Since each join vertex of $T$ has been split into a source and a sink in $N$, no internal vertex of $P$ is obtained by splitting a join vertex of
$T$. Let $P'$ arise from $P$ by contracting all arcs in $P \cap A''$. Then $P'$ remains an $s_i$-$t_j$ path. It follows that arcs in $\{\eta(a): a\in P'\} 
=\{\eta(a): a\in P \cap A'\}$ form a $v_i$-$v_j$ path $Q$ in $T$ that does not contain any join vertex other than $v_i$ and $v_j$. By definition, 
$Q\in \mathcal{Q}$.

Conversely, let $R$ consist of all arcs of $Q$ outside the set $\{a_kb_k: 1\le k\le m-1\ {\rm with}\ |B_k|=4\}$. Then there is a unique arc $e'$ of $N$ with $\eta(e')=e$ 
for each $e\in R$; set $R':=\{e': e\in R\}$. For $i\le k\le j$, if $|B_k|=3$ and $a_k\in V(Q)$, add $a_{k1}a_{k2}$ to $R'$;  if $|B_k|=4$ and $a_kb_k \in Q$, add both $a_{k,1}b_{k1}$
and $b_{k1}b_{k2}$ to $R'$;  if $|B_k|=4$, $a_kb_k \notin Q$, and $u_k \in \{a_k,b_k\} \cap V(Q)$, add $u_{k1}u_{k2}$ to $R'$. From the construction of $N$, we
see that the arcs in the resulting $R'$ form an $s_i$-$t_j$ path $P$ in $N$, such that the arcs in $\{\eta(a): a\in P \cap A' \}$ form $Q$.  \qed\\

The following three lemmas relate $(\Lambda, \Pi)$-cuts in $N$ to FAS's in $T$.

\begin{lemma}\label{lem:cut to FAS}
Let $Z$ be an inclusionwise minimal $(\Lambda, \Pi)$-cut in $N$ with $Z\cap A''=\emptyset$, and let $F:=\{\eta(a):\ a\in Z\}$. Then $F$ is an FAS in $T$ with $w(F)\le c(Z)$. 
Furthermore, if $F$ contains no horizontal cut of $T$, then $w(F)=c(Z)$.
\end{lemma}
\vspace{-1mm}

{\bf Proof.}  From the definitions of the surjection $\eta: A' \rightarrow A$ and arc capacities in $N$, it follows instantly that  
$w(F)\le \sum_{a\in Z} c(\eta(a))=c(Z)$. To prove that $F$ is an FAS of $T$, assume the contrary: $C\cap F=\emptyset$ for some cycle $C$ in $T$. Let 
$v_i$ (resp. $v_j$) be the join vertex on $C$ with the smallest (resp. largest) subscript and let $Q:=C[v_i, v_j]$ if $i<j$ and $Q:=C$ otherwise. 
By Lemma \ref{lem:obs on cycles}(iv), we have $Q \in \mathcal{Q}$. By Lemma \ref{lem:btw paths}, there exists an $s_i$-$t_j$ path $P$ in $N$ with 
$\{\eta(a): a\in P \cap A'\}=Q$. Since $Q \cap F=\emptyset$, we have  $\{\eta(a): a\in P \cap A'\} \cap \{\eta(a):\ a\in Z\} =\emptyset$, so
$(P \cap A')\cap Z=\emptyset$. Since $Z\cap A''=\emptyset$, we obtain $(P \cap A'')\cap Z=\emptyset$. Hence $P\cap Z=\emptyset$, contradicting 
the hypothesis that $Z$ is a $(\Lambda, \Pi)$-cut in $N$.

Assume that $F$ contains no horizontal cut of $T$. If $Z$ does not contain $\{a_{k1}b_{k1}, a_{k2}b_{k2}\}$ for any $1\le k\le m-1$, then there is 
a bijection between $Z$ and $F$ and hence $w(F)=c(Z)$. So we further assume that $\{a_{k1}b_{k1}, a_{k2}b_{k2}\} \subseteq Z$ for some $1\le k\le m-1$.
Since $Z$ is an inclusionwise minimal $(\Lambda, \Pi)$-cut in $N$, there exists an $s_p$-$a_{k1}$ path $P_1$ and a $b_{k2}$-$t_q$ path $P_2$ in $N$, such that 
$P_1\cap Z=\emptyset=P_2\cap Z$ for some $s_p\in \Lambda$ and $t_q\in \Pi$. Thus $\{a_{k1}t_k, s_{k+1}b_{k2}\} \subseteq Z$, 
because both $P_1\cup\{a_{k1}t_k\}$ and $P_2\cup\{s_{k+1}b_{k2}\}$ are paths from $\Lambda$ to $\Pi$ and hence must intersect $Z$. Therefore
the horizontal cut $\{a_kv_k, a_kb_k, v_{k+1}b_k\}$ of $T$ is fully contained in $F$; this contradiction completes the proof. \qed

\vskip 4mm

\begin{lemma}\label{lem:FAS to cut}
Let $F$ be an inclusionwise minimal FAS in $T$ containing no horizontal cut. Then there exists a $(\Lambda, \Pi)$-cut $Z$ in $N$ such that $F=\{\eta(a): a\in Z\}$ and $w(F)=c(Z)$.
\end{lemma}
\vspace{-1mm}

{\bf Proof.} Since $F$ contains no horizontal cut of $T$, $v_{i+1}v_i\notin F$ for all $i$ with $|B_i|\le 3$ by Lemma \ref{lem:FAS with no horizontal cut}(v). From the 
construction of $N$, we see that 

(1) for each $a\in F$, there is an arc $a'$ in $N$ with $\eta(a')=a$.  

Let $J:=\{k: |B_k|=4 \,\, \mbox{and} \,\, a_kb_k\in F\}$. Then 

(2) for each $a\in F\de\{a_kb_k:k\in J\}$, there is a unique arc $a'$ in $N$ with $\eta(a')=a$. 

By Lemma \ref{lem:FAS with no horizontal cut}(iii), exactly one of $a_kv_k$ and $v_{k+1}b_k$ belongs to $F$ for each $k\in J$. Let $J_1:=\{k\in J:a_kv_k\in F\}$ and 
$J_2:=\{k\in J:v_{k+1}b_k\in F\}$. Define 
\[Z:=\{a': a\in F\de\{a_kb_k:k\in J\}\}\cup \{a_{k1}b_{k1}: k\in J_1\}\cup\{a_{k2}b_{k2}: k\in J_2\}.\] 
Then $F=\{\eta(a): a\in Z\}$ and $w(F)=c(Z)$ by (1), (2) and the definition of $c$. It remains to show that $Z$ is a $(\Lambda, \Pi)$-cut in $N$.

For this purpose, let $P$ be an arbitrary path from a source $s_i\in \Lambda$ to a sink $t_j\in \Pi$ in $N$. We aim to prove that

(3) $P \cap Z \ne \emptyset$. 

By Lemma \ref{lem:paths in tildeT}, we have $i \le j$. By Lemma \ref{lem:btw paths}, the arcs in $\{\eta(a): a\in P \cap A'\}$ form a $v_i$-$v_j$ 
path $Q \in {\cal Q}$. By Lemma \ref{lem:FAS with no horizontal cut}(ii), we obtain

(4) $Q\cap F=\{\eta(a): a\in P \cap A'\} \cap F \neq \emptyset$.

We may assume that 

(5) $Q \cap F \subseteq \{a_kb_k: k\in J\}$ and hence $i=j$ by Lemma \ref{lem:FAS with no horizontal cut}(iv).

Otherwise, let $e \in (Q\cap F) \de \{a_kb_k: k\in J\}$. By (2), there is a unique arc $e'$ in $N$ with $\eta(e')=e$. Then $e' \in P\cap Z$ and hence
(3) holds. 

Since $i=j$, from the construction of $N$ it follows that 

(6) if $P \cap \{a_{k1}b_{k1}, a_{k2}b_{k2}\} \ne \emptyset$, then $k \in \{i, i-1\}$.  

Symmetry and (6) allows us to assume that $P$ is as described in one of the following two cases.

{\bf Case 1.} $P=R\cup \{a_{i1}b_{i1}, b_{i1}t_i\}$, where $R$ is an $s_i$-$a_{i1}$ path in $N$, but $P$ contains neither $a_{(i-1)1}b_{(i-1)1}$ nor
$a_{(i-1)2}b_{(i-1)2}$. In this case, $Q \cap F=\{a_ib_i\}$ by (4)-(6) and hence $\{\eta(a): a\in R \cap A'\} \cap F=\emptyset$.
Let $P':=R\cup \{a_{i1}t_i\}$ and let $Q':=\{\eta(a): a\in P' \cap A'\}$. By Lemma \ref{lem:btw paths}, $Q'$ is a cycle in ${\cal Q}$. Since $F$ is 
an FAS of $T$, we have $Q' \cap F \ne \emptyset$, so $a_iv_i=\eta(a_{i1}t_i) \in F$ (as $Q \cap F=\{a_ib_i\}$). By the definition of $Z$, we obtain 
$a_{i1}b_{i1} \in Z$. Hence $a_{i1}b_{i1} \in P\cap Z$. This proves (3).

{\bf Case 2.} $P=\{s_ia_{(i-1)2},a_{(i-1)2}b_{(i-1)2}\} \cup R \cup \{a_{i1}b_{i1}, b_{i1}t_i\}$, where $R$ is a $b_{(i-1)2}$-$a_{i1}$ path in $N$. By (4) and (5), 
at least one of $a_{i-1}b_{i-1}$ and ${a_ib_i}$ belongs to $F$, say ${a_ib_i}\in F$. Consider the path $P'=\{s_ia_{(i-1)2},a_{(i-1)2}b_{(i-1)2}\} \cup R
\cup \{a_{i1}t_i\}$. By the same argument as used for Case 1, we deduce that $P' \cap Z \ne \emptyset$. Thus $[\{s_ia_{(i-1)2},a_{(i-1)2}b_{(i-1)2}\} \cup R]
\cap Z \ne \emptyset$ or $a_{i1}t_i \in Z$ (and thus $a_{i1}b_{i1} \in Z$). Consequently, (3) holds in either subcase. \qed

\vskip 4mm

\begin{lemma}\label{lem:if FAS has h cut}
Let $Z$ be a minimum $(\Lambda, \Pi)$-cut in $N$ and let $F:=\{\eta(a): a\in Z\}$. If $F$ contains a horizontal cut in $T$, then there exists a minimum 
FAS in $T$ that contains at least one horizontal cut.
\end{lemma}
\vspace{-1mm}

{\bf Proof.} Since $Z$ is a minimum $(\Lambda, \Pi)$-cut in $N$, clearly we may assume that it is inclusionwise minimal. Since each arc in $A''$ has an infinite 
capacity, $Z\cap A''=\emptyset$. By Lemma \ref{lem:cut to FAS}, $F$ is an FAS with $w(F)\le c(Z)$. According to Lemma \ref{lem:FAS to cut}, any FAS of $T$ with 
no horizontal cut has a total weight $\ge c(Z)$ (and hence $\ge w(F)$). Therefore there exists a minimum FAS in $T$ that contains a horizontal cut. \qed\\

Recall the terminology and notation introduced in the paragraph succeeding Theorem \ref{FAS}. Let ${\cal H}_i$ be the set of all horizontal cuts of $T$ contained 
in $B_i$ for $1\le i \le B_{m-1}$. Set ${\cal H}_0={\cal H}_m:=\emptyset$. For $\sigma \in {\cal H}_0$, set $T_{\sigma \ell}:=\emptyset$ and $T_{\sigma r}:=T$.  
For $\sigma \in {\cal H}_m$, set $T_{\sigma \ell}:=T$ and $T_{\sigma r}:=\emptyset$. For any $\sigma \in {\cal H}_i$ and $\omega \in {\cal H}_j$
with $0\le i <j \le m$, set $T_{\sigma \omega}: =T_{\sigma r} \cap T_{\omega \ell}$; that is, $T_{\sigma \omega}$ is the subtournament of $T$ induced by
all vertices in $V(T_{\sigma r}) \cap V(T_{\omega \ell})$. Observe that $T_{\sigma \omega}$ is also as depicted in Figure 3, whose first vertical block contains
all vertices in $B_{\sigma r} \cup A_{i+1}$ and last vertical block contains all vertices in $B_{\omega \ell} \cup A_j$. As defined in the first section, 
$\tau_w(T_{\sigma \omega})$ is the minimum total weight of an FAS in $T_{\sigma \omega}$.  Let  $\tau_w^*(T_{\sigma \omega})$ denote the minimum total weight 
of an FAS $F$ in $T_{\sigma \omega}$, subject to the constraint that $F$ contains no horizontal cut of $T$ in ${\cal H}_k$, for all $i<k<j$.

\vskip 3mm

{\bf Proof of Theorem \ref{FAS}.} For every $\sigma \in {\cal H}_i$ and $\omega \in {\cal H}_j$ with $0\le i <j \le m$,   
Lemma \ref{lem:FAS} leads to the following recurrence relation:
\[\tau_w(T_{\sigma \omega})=\min \{\tau_w^*(T_{\sigma \omega}), \hskip 2mm \min_{\pi}\{\tau_w(T_{\sigma \pi})+w(\pi)+\tau_w(T_{\pi \omega})\}\},\]
where $\pi$ ranges over horizontal cuts of $T$ contained in ${\cal H}_k$, for all $i<k<j$.

Note that $v_{i+1}, v_{i+2}, \ldots, v_j$ are the only join vertices of $T_{\sigma \omega}$. Set $\Lambda_{\sigma \omega}:=\{s_{i+1},s_{i+2}, \ldots, s_j\}$ 
and $\Pi_{\sigma \omega}:=\{t_{i+1}, t_{i+2}, \ldots, t_j\}$. Let $(N_{\sigma \omega}, \Lambda_{\sigma \omega}, \Pi_{\sigma \omega}, c)$ be the  
corresponding network of $(T_{\sigma \omega}, w)$, which can be constructed in $O(|T_{\sigma \omega}|^2)$ time. 

Set $t:=j-i$. We shall find a minimum FAS $F_{\sigma \omega}$ in $T_{\sigma \omega}$ together with $\tau_w(T_{\sigma \omega})$ recursively, in the increasing order of $t$.

When $t=1$, the tournament $T_{\sigma \omega}$ contains only one vertical block and one join vertex $v_{i+1}$. Thus $\Lambda_{\sigma \omega}=\{s_{i+1}\}$ and 
$\Pi_{\sigma \omega}=\{t_{i+1}\}$. Let $Z_{\sigma \omega}$ be a minimum $s_{i+1}$-$t_{i+1}$ cut in $N_{\sigma \omega}$. Then $F_{\sigma \omega}:=\{\eta(a): a \in 
Z_{\sigma \omega}\}$ is clearly a minimum FAS in $T_{\sigma \omega}$ and $\tau_w(T_{\sigma \omega})=w(F_{\sigma \omega})$. We can find $Z_{\sigma \omega}$ (and hence 
$F_{\sigma \omega}$) in $O(|T_{\sigma \omega}|^3)$ time using Tarjan's algorithm \cite{tarjan2}. Since the sum of $|T_{\sigma \omega}|^3$ over all possible choices 
of the pair $(\sigma, \omega)$ is $O(n^3)$, where $n=|T|$, this step takes $O(n^3)$ time.

Suppose we have obtained $F_{\sigma \omega}$ and $\tau_w(T_{\sigma \omega})$ for all possible choices of the pair $(\sigma, \omega)$ corresponding to $t-1$. 
Let us proceed to the step for $t\ge 2$.

Let $\rho$ be a horizontal cut of $T$ contained in some ${\cal H}_k$, with $i<k<j$, that achieves $\min_{\pi}\{\tau_w(T_{\sigma \pi})+w(\pi)+
\tau_w(T_{\pi \omega})\}$; such $\rho$ can be found in $O(t^2)=O(|T_{\sigma \omega}|^2)$ time. Set $K_1:= F_{\sigma \rho} \cup \rho \cup F_{\rho \omega}$.

Let $Z$ be a minimum $\Lambda_{\sigma \omega}$-$\Pi_{\sigma \omega}$ cut in $N_{\sigma \omega}$, which can be found in $O(|T_{\sigma \omega}|^3)$ time
using aforementioned Tarjan's algorithm \cite{tarjan2}. Clearly, we may assume that $Z$ is inclusionwise minimal. Set $K_2:=\{\eta(a):a\in Z\}$. 

$\bullet$ If $K_2$ contains no horizontal cut of $T$, then $\tau_w^*(T_{\sigma \omega})=w(K_2)=c(Z)$ by Lemmas \ref{lem:cut to FAS} and \ref{lem:FAS to cut}. 
Let $w(K_i)= \min \{w(K_1), w(K_2)\}$, with $i=1$ or $2$. Set $F_{\sigma \omega}:=K_i$ and $\tau_w(T_{\sigma \omega}):=w(K_i)$.

$\bullet$ If $K_2$ contains a horizontal cut of $T$, then $\tau_w^*(T_{\sigma \omega})\ge \min_{\pi}\{\tau_w(T_{\sigma \pi})
+w(\pi)+\tau_w(T_{\pi \omega})\}$ by Lemma \ref{lem:if FAS has h cut}. Set $F_{\sigma \omega}:=K_1$ and $\tau_w(T_{\sigma \omega}):=w(K_1)$.

Since the sum of $|T_{\sigma \omega}|^3$ over all possible choices of the pair $(\sigma, \omega)$ is $O(tn^3)$, where $n=|T|$, this step takes $O(n^4)$ time
as $1\le t \le n$. Therefore, a minimum FAS in $T$ and $\tau_w(T)$ can be found in $O(n^5)$ time. \qed

\section{Cycle Packings}

In this section we devise a combinatorial polynomial-time algorithm for finding maximum cycle packings in arc-weighted tournaments in ${\cal G}^*$, which 
forms the backbone of the general algorithms to be presented in the subsequent section. 

\begin{theorem}\label{CP}
Let $T=(V,A)$ be a tournament in ${\cal G}^*$ with a nonnegative integral weight $w(e)$ on each arc $e$. Then a maximum cycle packing in $(T,w)$
can be found in $O(n^7)$ time, where $n=|V|$.
\end{theorem}

We point out that a {\em cycle packing} actually consists of a family of cycles $C_1, C_2, \ldots, C_t$ together with a family of nonnegative integers 
(multiplicities) $y(C_1), y(C_2), \ldots, y(C_t)$  such that $\sum_{e \in C_i} y(C_i) \le w(e)$ for each $e\in A$. It is called {\em maximum} if its size 
$\sum_{i=1}^t y(C_i)$ is as large as possible. To see that these concepts are essentially the same as those introduced in Section 1, define $y(C)=0$ for 
any other cycle $C$ in $T$. Furthermore, to ensure efficiency of an algorithm for the cycle packing problem, $t$ must be bounded above by a polynomial in $n$. 

By Theorem \ref{Theorem1}, the maximum size of a cycle packing is equal to the minimum total weight of an FAS in $T$; this min–max relation will 
serve as an optimality criterion in our algorithm design. A natural question then arises: Can such an algorithm be designed without relying on 
this criterion? We have devoted considerable effort to pursuing this approach, but so far without success. These attempts suggest that such an approach 
might be very difficult to carry out, particularly because the cycle packing problem is a type of multicommodity flow problem (see, for instance,
Even, Naor, and Zosin \cite{ENZ}).

To fully leverage this optimality criterion, we need to clean up $T$ beforehand. 

\begin{lemma}\label{preprocessing}
The input tournament $T$ can be preprocessed in $O(n^7)$ time so that it satisfies the following properties:
\begin{itemize}
\vspace{-2mm}
\item[(i)] Each horizontal block $B_i$ has precisely four vertices; and
\vspace{-2mm}
\item[(ii)] Each arc $e$ with $w(e)>0$ is contained in a minimum FAS in $(T, w)$.
\end{itemize}
\end{lemma} 

\vspace{-1mm}

{\bf Proof.} (i) Let us apply the following operations to each horizontal block $B_i$ with $|B_i|\le 3$.

When $|B_i|=2$, we first delete the arc $v_{i+1}v_i$. Then add two vertices $a_i$ and $b_i$ and six arcs $a_iv_i, v_{i+1}a_i, b_iv_i, v_{i+1}b_i, a_ib_i, v_iv_{i+1}$, 
define $w(a_iv_i)=w(v_{i+1}a_i):=w(v_{i+1}v_i)$ and $w(b_iv_i)=w(v_{i+1}b_i)=w(a_ib_i)=w(v_iv_{i+1}):=0$. Finally, for each vertex $u\in (\cup_{t=1}^i A_t)\cup
(\cup_{t=1}^{i-1}B_t)\de v_i$, add two arcs $ua_i$ and $ub_i$ with zero weights; for each vertex $u\in (\cup_{t=i+1}^m A_t)\cup(\cup_{t=i+1}^{m-1}B_t)\de v_{i+1}$, 
add two arcs $a_iu$ and $b_iu$ with zero weights. Replace $B_i$ by the subtournment with four vertices $v_i, v_{i+1}, a_i, b_i$.

When $|B_i|=3$, we first delete the arc $v_{i+1}v_i$. Then add one vertex $b_i$ and four arcs $b_iv_i, v_{i+1}b_i, a_ib_i, v_iv_{i+1}$, define $w(b_iv_i)=w(v_{i+1}b_i)
:=w(v_{i+1}v_i)$ and $w(a_ib_i)=w(v_iv_{i+1}):=0$. Finally, for each vertex $u \in (\cup_{t=1}^i A_t)\cup(\cup_{t=1}^{i-1}B_t)\de v_i$, add an arc $ub_i$ with zero 
weight; for each vertex $u \in (\cup_{t=i+1}^m A_t)\cup(\cup_{t=i+1}^{m-1} B_t)\de v_{i+1}$, add an arc $b_iu$ with zero weight. Replace $B_i$ by the subtournment 
with four vertices $v_i, v_{i+1}, a_i, b_i$.

Clearly, these operations preserve both the minimum total weight of an FAS and the maximum size of a cycle packing, and they induce a one-to-one correspondence 
between cycle packings in the original and those in the resulting $T$.   

(ii) Consider an arbitrary arc $e$ of $T$ with $w(e)>0$. Let $c$ be the weight function obtained from $w$ by replacing $w(e)$ with zero. Let $\tau_w(T)$ 
(resp. $\tau_c(T)$) be the minimum total weight of an FAS in $(T, w)$ (resp. $(T, c)$), which can be found in $O(n^5)$ time by Theorem \ref{FAS}. 
Let $w'$ arise from $w$ by replacing $w(e)$ with $\tau_w(T)-\tau_c(T)$. Clearly, $\tau_{w'}(T)=\tau_w(T)$ and $e$ is contained in a minimum FAS in $(T, w')$ 
if $w'(e)>0$. Replace $w$ by $w'$ and repeat the process until all arcs with positive weights have been scanned; the resulting 
weight function $w$ will satisfy property (ii). \qed\\

So we may assume hereafter that $T$ and $w$ are as described in the above lemma. We shall use the same notation as introduced in Section 1 (with $T$ in
place of $G$). 

\begin{lemma}\label{lem: saturated arcs}
Let $y$ be a maximum cycle packing in $(T, w)$. Then 
\[\sum_{e \in C\in \mathcal{C}} y(C)=w(e) \hskip 2mm \mbox{for each}  \hskip 2mm e\in A. \]
\end{lemma}
\vspace{-1mm}

{\bf Proof.} By the definition of cycle packing, $\sum_{e\in C\in\mathcal{C}} y(C)\le w(e)$ for each $e\in A$. Assume the contrary: $\sum_{a\in C\in\mathcal{C}} y(C)< w(a)$ for 
some $a\in A$. Then $\sum_{a\in C\in\mathcal{C}}y(C)\le w(a)-1$, as $y$ is integral. Let $c$ be the weight function obtained from $w$ by replacing
$w(a)$ with $w(a)-1$. Then $y$ is also a maximum cycle packing in $T$ with respect to $c$. It follows that

(1) $\nu_c(T)=\nu_w(T)$. 

By Lemma \ref{preprocessing}(ii), $a$ is contained in a minimum FAS $F$ in $(T, w)$. Since $\tau_w(T)=w(F)$ and $\tau_c(T)\le c(F) < w(F)$,
we obtain

(2) $\tau_c(T)<\tau_w(T)$. 

By Theorem \ref{Theorem1}, $T$ is cycle Mengerian. Hence $\nu_c(T)=\tau_c(T)$ and $\nu_w(T)=\tau_w(T)$, which together with (1) yield
$\tau_c(T)=\tau_w(T)$, contradicting (2). \qed\\ 

For each cycle $C$ in $T$, let $v_i$ (resp. $v_j$) be the join vertex on $C$ with the smallest (resp. largest) subscript. 
By Lemma \ref{lem:obs on cycles}(ii) and (iv), if $i<j$, then $C[v_j, v_i]$ is fully contained in the horizontal blocks $B_{j-1}, B_{j-2}, \ldots, B_i$ and 
passes through the join vertices $v_j,v_{j-1}, \ldots,v_i$ in order. Moreover, $C[v_i, v_j]$ contains no join vertex other than $v_i, v_j$ and contains at 
most one non-join vertex from each horizontal block $B_k$ for $i \le k \le j-1$. If $i=j$, set $C[v_i, v_j]:=C$ and $C[v_j, v_i]:=v_i$. In both cases, 
we call $C[v_i, v_j]$ the {\em lower segment} of $C$, call $C[v_j, v_i]$ the {\em upper segment} of $C$, and denote them by $L(C)$ and $U(C)$, respectively.

To find a maximum cycle packing in $(T, w)$, we shall actually first find a maximum lower segment packing and then complete these lower segments 
into cycles by adding corresponding upper segments. The crux of our algorithm is to determine the portion of $w(e)$, for each arc $e$ contained in horizontal 
blocks, that is used to pack these lower segments. This motivates us to consider the following six families of cycles in $T$ for each $1 \le i \le m-1$:

\vskip 1mm
$\bullet$ $\mathcal{C}_i^1$ consists of cycles whose lower segments end at $v_i$ and contain $a_iv_i$;

\vskip 1mm
$\bullet$ $\mathcal{C}_i^2$ consists of cycles whose lower segments end at $v_i$, contain $b_iv_i$ but do not contain $a_ib_i$;

\vskip 1mm   
$\bullet$ $\mathcal{C}_i^3$ consists of cycles whose lower segments end at $v_i$ and contain both $a_ib_i$ and $b_iv_i$;

\vskip 1mm 
$\bullet$ $\mathcal{C}_i^4$ consists of cycles whose lower segments start at $v_{i+1}$ and contain $v_{i+1}b_i$;

\vskip 1mm  
$\bullet$ $\mathcal{C}_i^5$ consists of cycles whose lower segments start at $v_{i+1}$, contain $v_{i+1}a_i$ but do not contain 

\hskip 8mm $a_ib_i$; and

\vskip 1mm
$\bullet$ $\mathcal{C}_i^6$ consists of cycles whose lower segments start at $v_{i+1}$ and contain both $v_{i+1}a_i$ and $a_ib_i$.
\vskip 1mm

Note that there might exist cycles in $T$ whose lower segments end at $v_i$ (or start at $v_{i+1}$) but contain 
neither $a_i$ nor $b_i$; all of them are outside these six families.  For each cycle packing $y$ in $(T, w)$, 
$1\le i\le m-1$ and $1\le j\le 6$, set 

\vskip 2mm
$\bullet$ $y_i^j:=\sum_{C\in\mathcal{C}_i^j} y(C)$.
\vskip 2mm

For convenience, we also set

\vskip 2mm
$\bullet$  $V_i^{\ell}:= \mbox{the vertex set of} \hskip 2mm (\cup_{t=1}^i A_t)\cup(\cup_{t=1}^{i-1}B_t) \de v_i$ and 

\vskip 2mm
$\bullet$ $V_i^r:= \mbox{the vertex set of} \hskip 2mm (\cup_{t=i+1}^{m} A_t)\cup(\cup_{t=i+1}^{m-1} B_t) \de v_{i+1}$. 

\vskip 2mm

\begin{lemma}\label{lem:proper packing}
There is a maximum cycle packing $y$ in $(T, w)$ that satisfies the following equalities for all $1\le i\le m-1$:
\begin{itemize}
\vspace{-1.5mm}
\item [(i)] $y_i^1=\min\{w(a_iv_i),\ \sum_{u \in V_i^{\ell}} \, w(ua_i)\}$;
\vspace{-1.5mm}
\item [(ii)] $y_i^2=\min\{w(b_iv_i),\ \sum_{u\in V_i^{\ell}} \, w(ub_i)\}$;
\vspace{-1.5mm}
\item [(iii)] $y_i^3=\min\{w(a_ib_i),\ \sum_{u\in V_i^{\ell}} \, w(ua_i)-y_i^1,\ w(b_iv_i)-y_i^2\}$;
\vspace{-1.5mm}
\item[(iv)] $y_i^4=\min\{w(v_{i+1}b_i),\ \sum_{u \in V_i^r} \, w(b_iu)\}$;
\vspace{-1.5mm}
\item [(v)] $y_i^5=\min\{w(v_{i+1}a_i),\ \sum_{u\in V_i^r} \, w(a_i u)\}$; and 
\vspace{-1.5mm}
\item [(vi)] $y_i^6=\min\{w(a_ib_i)-y_i^3,\ \sum_{u\in V_i^r} \, w(b_iu)- y_i^4,\ w(v_{i+1}a_i)-y_i^5\}$.
\end{itemize}
\end{lemma}

{\bf Proof.} Let $y$ be a maximum cycle packing in $(T, w)$ such that 

\vskip 2mm
(1) the tuple $(y_1^1, y_1^2, y_1^3, y_1^4, y_1^5, y_1^6, y_2^1, y_2^2, y_2^3, y_2^4, y_2^5, y_2^6, \ldots, y_{m-1}^1, y_{m-1}^2, y_{m-1}^3, y_{m-1}^4, y_{m-1}^5, y_{m-1}^6)$
\vskip 2mm

\noindent is maximum in lexicographic order.

Let us show that $y$ is as desired. Clearly, $y_i^j$ is bounded above by the right hand side as specified in the lemma for all 
$1\le i\le m-1$ and $1\le j\le 6$. Assume on the contrary that 

\vskip 1mm
(2) the strict inequality holds for some $y_i^j$; subject to this, $(i,j)$ is lexicographically minimum. 

\vskip 1mm
Depending on the value of $j$, we proceed by considering six cases, and shall make extensive use of cross-free techniques in our proof.

\vskip 1mm
{\bf Case 1}. $j=1$; that is, $y_i^1<\min\{w(a_iv_i),\ \sum_{u \in V_i^{\ell}} \, w(ua_i)\}$.

\vskip 1mm

In this case, Lemma \ref{lem: saturated arcs} guarantees the existence of two cycles $C_1$ and $C_2$ in $\mathcal{C}\de\mathcal{C}_i^1$, such
that 

$\bullet$ $y(C_1)>0$, $a_iv_i\in U(C_1)$, and

$\bullet$ $y(C_2)>0$, $a_iv_i\notin C_2$, $ua_i\in C_2$ for some $u\in V_i^{\ell}$.   

\noindent By Lemma \ref{lem:obs on cycles}(i) and (ii), $C_2$ must pass through $v_i$, and all vertices on $C_2(v_i, u]$ are contained in $V_i^{\ell}$.
Let $C_1'$ be the cycle $C_2[v_i,a_i]\cup\{a_iv_i\}$. Clearly, $C_1' \in \mathcal{C}_i^1$. Let $v_t$ be the join vertex contained in $C_1$ with
the smallest subscript, and let $C_2'$ be a cycle in $C_1[v_i, a_i]\cup C_2[a_i,v_i]$ that contains $C[v_i, v_t]$. Let $\bar{y}$ be the maximum cycle packing
obtained from $y$ by decreasing both $y(C_1)$ and $y(C_2)$ by $1$ and increasing both $y(C_1')$ and $y(C_2')$ by $1$. Then $\bar{y}_p^q = y_p^q$
for all $(p,q)$ lexicographically smaller than $(i,1)$, while $\bar{y}_i^1 = y_i^1+1$. Thus the assumption (1) on $y$ is violated by $\bar{y}$.

\vskip 1mm
{\bf Case 2}. $j=2$; that is, $y_i^2<\min\{w(b_iv_i),\ \sum_{u\in V_i^{\ell}} \, w(ub_i)\}$.

\vskip 1mm

In this case, Lemma \ref{lem: saturated arcs} guarantees the existence of two cycles $C_1$ and $C_2$ in $\mathcal{C}\de\mathcal{C}_i^2$, such
that 

$\bullet$ $y(C_1)>0$, $b_iv_i\in U(C_1)$, and

$\bullet$ $y(C_2)>0$, $b_iv_i\notin C_2$, $ub_i\in C_2$ for some $u\in V_i^{\ell}$.   

\noindent By Theorem \ref{structure}, we have $a_iv_i \in C_2$. By Lemma \ref{lem:obs on cycles}(i) and (ii), $C_2$ must
pass through $v_i$, and all vertices on $C_2(v_i, u]$ are contained in $V_i^{\ell}$. Let $C_1'$ be the cycle $C_2[v_i,b_i]\cup\{b_iv_i\}$. 
Clearly, $C_1' \in \mathcal{C}_i^2$. Let $v_t$ be the join vertex contained in $C_1$ with the smallest subscript, and let $C_2'$ be a cycle in 
$C_1[v_i, b_i]\cup C_2[b_i,v_i]$ that contains $C[v_i, v_t]$. Let $\bar{y}$ be the maximum cycle packing
obtained from $y$ by decreasing both $y(C_1)$ and $y(C_2)$ by $1$ and increasing both $y(C_1')$ and $y(C_2')$ by $1$. Then $\bar{y}_p^q = y_p^q$
for all $(p,q)$ lexicographically smaller than $(i,2)$, while $\bar{y}_i^2 = y_i^2+1$. Thus the assumption (1) on $y$ is violated by $\bar{y}$.

\vskip 1mm
{\bf Case 3}. $j=3$; that is, $y_i^3<\min\{w(a_ib_i),\ \sum_{u\in V_i^{\ell}} \, w(ua_i)-y_i^1,\ w(b_iv_i)-y_i^2\}$.
\vskip 1mm

In this case, Lemma \ref{lem: saturated arcs} guarantees the existence of two cycles $C_1 \in \mathcal{C}\de \mathcal{C}_i^3$ and $C_2 \in 
\mathcal{C}\de (\mathcal{C}_i^1 \cup \mathcal{C}_i^3)$, such that 

$\bullet$ $y(C_1)>0$, $a_ib_i\in C_1$, and

$\bullet$ $y(C_2)>0$, $ua_i\in C_2$ for some $u\in V_i^{\ell}$.   

\noindent  Since $C_1\notin \mathcal{C}_i^3$, it must pass through $v_{i+1}$ (and hence $v_{i+1}a_i$) by Lemma \ref{lem:obs on cycles}(i) and (ii).
Since $C_2\notin \mathcal{C}_i^1\cup\mathcal{C}_i^3$, it contains neither $a_iv_i$ nor $a_ib_i$, which 
implies that $a_iv\in C_2$ for some $v\in V_i^r$.  Hence $C_2$ also passes through $v_{i+1}$ and therefore traverses the path $v_{i+1}b_iv_i$. Let $C_1'$ 
be the cycle $C_1[b_i, v_{i+1}] \cup \{v_{i+1}b_i\}$, let $C_2'$ be the cycle $C_2[b_i, a_i] \cup\{a_ib_i\}$, and let $C_3'$ be the cycle $C_2[a_i,v_{i+1}] 
\cup\{v_{i+1}a_i\}$. Let $\bar{y}$ be obtained from $y$ by decreasing both $y(C_1)$ and $y(C_2)$ by $1$ and increasing all of $y(C_1'), y(C_2')$ and 
$y(C_3')$ by $1$. Then $\bar{y}$ would be a cycle packing with size greater than that of $y$, contradicting the maximality assumption on $y$.   
 
\vskip 1mm
{\bf Case 4}. $j=4$; that is, $y_i^4<\min\{w(v_{i+1}b_i),\ \sum_{u \in V_i^r} \, w(b_iu)\}$.
\vskip 1mm

This case is symmetric to Case 1, so the proof proceeds along the same lines.
 
\vskip 1mm
{\bf Case 5}. $j=5$; that is, $y_i^5<\min\{w(v_{i+1}a_i),\ \sum_{u\in V_i^r} \, w(a_i u)\}$.
\vskip 1mm

This case is symmetric to Case 2, so the proof proceeds along the same lines.

\vskip 1mm
{\bf Case 6}. $j=6$; that is, $y_i^6<\min\{w(a_ib_i)-y_i^3,\ \sum_{u\in V_i^r} \, w(b_iu)- y_i^4,\ w(v_{i+1}a_i)-y_i^5\}$.
\vskip 1mm

In this case, Lemma \ref{lem: saturated arcs} guarantees the existence of two cycles $C_1 \in \mathcal{C}\de (\mathcal{C}_i^3 \cup \mathcal{C}_i^6)$ 
and $C_2 \in \mathcal{C}\de (\mathcal{C}_i^4 \cup \mathcal{C}_i^6)$, such that 

$\bullet$ $y(C_1)>0$, $a_ib_i\in C_1$, and

$\bullet$ $y(C_2)>0$, $b_iu\in C_2$ for some $u\in V_i^r$.   

\noindent By Lemma \ref{lem:obs on cycles}(iv), $a_ib_i\in U(C_1)$, so $C_1$ contains the path $v_{i+1}a_ib_iv_i$. Since $C_2\notin \mathcal{C}_i^4
\cup\mathcal{C}_i^6$, it contains neither $v_{i+1}b_i$ nor $a_ib_i$. Hence $vb_i \in C_2$ for some $v\in V_i^{\ell}$. Therefore $C_2$ must pass through both 
$v_i$ and $v_{i+1}$. Let $C_1'$ be the cycle $C_1[v_i, v_{i+1}] \cup\{v_{i+1}a_i, a_iv_i\}$, let $C_2'$ be the cycle $C_2[v_i,b_i] \cup \{b_iv_i\}$, and
let $C_3'$ be the cycle $C_2[b_i, v_{i+1}] \cup \{v_{i+1}a_i,a_ib_i\}$. Let $\bar{y}$ be obtained from $y$ by decreasing both $y(C_1)$ and $y(C_2)$ by $1$ 
and increasing all of $y(C_1'), y(C_2')$ and $y(C_3')$ by $1$. Then $\bar{y}$ would be a cycle packing with size greater than that of $y$, contradicting 
the maximality assumption on $y$.   \qed\\
 
For each vertex $v$ of $T$, let $\delta^-(v)$ (resp. $\delta^+(v)$) be the set of all arcs entering (resp. leaving) $v$. For each
join vertex $v_i$ of $T$, set $Z_ i^-:= \delta^-(v_i)\de \{a_iv_i, b_iv_i\}$ and $Z_ i^+:= \delta^+(v_i)\de \{v_ia_{i-1}, v_ib_{i-1}\}$. Note that 
$Z_ 1^+= \delta^+(v_1)$ and $Z_ m^-= \delta^-(v_m)$. Let $\mathcal{C}_i^-$ (resp. $\mathcal{C}_i^+$) be the set of cycles in $T$ whose lower 
segments end (resp. start) at $v_i$. 

\begin{lemma}\label{zang1}
Let $y$ be a maximum cycle packing in $(T, w)$ as specified in Lemma \ref{lem:proper packing}. Then the following statements hold
for all $1\le i\le m$:
\begin{itemize}
\vspace{-1mm}
\item[(i)] $\sum_{C\in\mathcal{C}_i^-} y(C)= w(Z_i^-) + \sum_{j=1}^3y_i^j$ and
\vspace{-2mm}
\item[(ii)] $\sum_{C\in\mathcal{C}_i^+} y(C)= w(Z_i^+) + \sum_{j=4}^6y_{i-1}^j$,
\vspace{-1mm}
\end{itemize}
\noindent where $y_0^j=0$ for $1\le j \le 6$.
\end{lemma}

\vspace{-1mm}

{\bf Proof.} Note that each cycle in $\mathcal{C}_i^-$ contains an arc in $\delta^-(v_i)$ and that arcs in $Z_i^-$ 
are contained only in cycles in $\mathcal{C}_i^-$. By Lemma \ref{lem: saturated arcs} and the definition of $y_i^j$, we obtain (i).
Similarly, (ii) holds. \qed\\

Now we are ready to present a combinatorial polynomial-time algorithm for finding a maximum cycle packing in $(T, w)$.
\vskip 2mm

{\bf Proof of Theorem \ref{CP}.} Let $y$ be a maximum cycle packing in $(T, w)$ as specified in Lemma \ref{lem:proper packing}. 
Although we have not yet given any algorithm for finding $y$, the values of $y_i^1, y_i^2, y_i^3, y_i^4, y_i^5, y_i^6$ have been determined exactly for all 
$1\le i \le m-1$, which allow us to devise an efficient algorithm for finding a maximum lower segment packing in $(T,w)$. 

Let $(N, \Lambda, \Pi, c)$ be the corresponding network of $(T,w)$, and let $c'$ be the following new capacity function defined on
$N$:  

\vskip 1mm
$\bullet$ $c'(a_{i1}t_i):=y_i^1$, 

\vskip 1mm
$\bullet$ $c'(b_{i1}t_i):=y_i^2+y_i^3$,

\vskip 1mm
$\bullet$ $c'(a_{i1}b_{i1}):=y_i^3$,

\vskip 1mm
$\bullet$ $c'(s_{i+1}b_{i2}):=y_i^4$,

\vskip 1mm
$\bullet$ $c'(s_{i+1}a_{i2}):=y_i^5+y_i^6$, and

\vskip 1mm
$\bullet$ $c'(a_{i2}b_{i2}):=y_i^6$
\vskip 1mm

\noindent for all $1\le i \le m-1$, and $c'(a)=c(a)$ for any other arc $a$. We use ${\cal P}$ to denote the set of all simple paths in $N$ 
from $\Lambda$ to $\Pi$.

\vskip 1mm
(1) The maximum value of a $(\Lambda, \Pi)$-flow in $(N, c')$ is at least $\sum_{C\in\mathcal{C}} y(C)$.

\vskip 1mm

To justify this, let $Q_C$ denote the lower segment of a cycle $C$ in $T$. By Lemma \ref{lem:btw paths}, there exists an path $P_C$ from $\Lambda$ to $\Pi$ 
in $N$, such that the arcs in $\{\eta(a): a\in P_C \cap A' \}$ form $Q_C$ (recall the construction of $N$). Define $g(P):=y(C)$ if 
$P=P_C$ for a cycle $C$ in $T$ and $g(P):=0$ for all other paths $P$ in ${\cal P}$. Clearly,  $\sum_{e \in P \in {\cal P}}\, g(P) \le  c(e)$ for all 
arcs $e$ of $N$. So $g$ is a $(\Lambda, \Pi)$-flow (in the path packing form) with value $\sum_{P \in {\cal P}}\, g(P)=\sum_{C\in\mathcal{C}} y(C)$. 
This proves (1).

\vskip 1mm

Throughout the remainder of this proof, let $\beta(v)$ denote the set of all arcs incident with a vertex $v$ in $N$. For each vertex subset $U$ of $N$, set 
$\beta(U):=\cup_{v\in U} \beta(v)$.  

\vskip 1mm

(2) $c'(\beta(t_i))=\sum_{C\in\mathcal{C}_i^-} y(C)$ and $c'(\beta(s_i))=\sum_{C\in\mathcal{C}_i^+} y(C)$ for $1\le i \le m$.

Recall the construction of $N$, each arc $e$ in $Z_i^-$ corresponds to a unique arc $e'$ in $N$ and $c(e')=w(e)$. Set $Y_i:=\{e': e\in Z_i^-\}$. 
Then $\beta(t_i)=Y_i \cup \{a_{i1}t_i, b_{i1}t_i\}$, which is the set of all arcs entering $t_i$ in $N$. By Lemma \ref{lem:proper packing} and 
the definition of $c'$, we have $c'(\beta(t_i))= c'(Y_i)+c'(a_{i1}t_i)+ c'(b_{i1}t_i)=c(Z_i^-)+\sum_{j=1}^3y_i^j=w(Z_i^-) + \sum_{j=1}^3y_i^j$. So 
$c'(\beta(t_i))=\sum_{C\in\mathcal{C}_i^-} y(C)$ by Lemma \ref{zang1}(i). Similarly, we can prove that  $c'(\beta(s_i))=\sum_{C\in\mathcal{C}_i^+} 
y(C)$. So (2) is justified.

\vskip 1mm
(3) The minimum capacity of a $(\Lambda, \Pi)$-cut in $(N, c')$ is at most $\sum_{C\in\mathcal{C}} y(C)$, which is the capacity of both $\beta(\Lambda)$
and $\beta(\Pi)$. 

\vskip 1mm
Clearly, $\beta(\Pi)$ is a $(\Lambda, \Pi)$-cut in $N$, with capacity $\sum_{i=1}^{m-1} c'(\beta(t_i)) =\sum_{i=1}^{m-1} \sum_{C\in\mathcal{C}_i^-} 
y(C)=\sum_{C\in\mathcal{C}} y(C)$, where the first equality follows from (2). Similarly, $\beta(\Lambda)$ is a $(\Lambda, \Pi)$-cut in $N$, with 
capacity $\sum_{C\in\mathcal{C}} y(C)$. This establishes (3).

\vskip 1mm

From the max-flow min-cut theorem, we deduce that

\vskip 1mm
(4) both inequalities described in (1) and (3) are attained with equality. Furthermore, both $\beta(\Lambda)$ and $\beta(\Pi)$ are minimum 
$(\Lambda, \Pi)$-cuts in $(N, c')$ with capacity $\sum_{C\in\mathcal{C}} y(C)$. (So every maximum $(\Lambda, \Pi)$-flow in $(N, c')$ saturates 
every arc leaving $\Lambda$ or entering $\Pi$.) 

\vskip 1mm
(5) There is a maximum $(\Lambda, \Pi)$-flow $g: {\cal P} \rightarrow {\mathbb Z}_+$ (in the path packing form) in $(N, c')$,  such that, 
for any $P$ with $g(P)>0$ and $1\le i \le m-1$, if $a_{i1}b_{i1} \in P$, then $b_{i1}t_i \in P$; if $a_{i2}b_{i2} \in P$, then $s_{i+1}a_{i2} \in P$. 

\vskip 1mm
To justify this, let $f$ be an integral maximum $(\Lambda, \Pi)$-flow in $(N, c')$. By (4), $f$ saturates every arc leaving $\Lambda$ or entering 
$\Pi$. Using breadth-first search (see Tarjan \cite{tarjan} and Cormen et al. \cite{CL}) $O(n^2)$ times, we can construct a collection ${\cal R}$ of $O(n^2)$ 
directed paths $R$ in $N$ from $\Lambda$ to $\Pi$ together with integral multiplicities $h(R)>0$, such that  

\vskip 1mm
$\bullet$ $\sum_{a \in R \in {\cal R}}\, h(R) \le f(a)$ for each arc $a$ of $N$;
 
\vskip 1mm
$\bullet$ $\sum_{a \in R \in {\cal R}}\, h(R) = f(a)$ for each arc $a$ leaving $\Lambda$ or entering $\Pi$.
\vskip 1mm

The construction proceeds as follows: Find a directed path $R$ from $\Lambda$ to $\Pi$ in $N$ such that $f(a)>0$ for each arc $a$ on $R$ using breadth-first 
search, which runs in $O(n^2)$ time.  Define $h(R):=\alpha$, where $\alpha=\min \{f(a): a \in R\}$. Replacing $f(a)$ by $f(a)-\alpha$ for 
each arc $a$ on $R$, and repeat the process until ${\rm val}(f)=0$. This step requires $O(n^4)$ time. 

Suppose there is a path $R$ in ${\cal R}$ with $h(R)>0$ such that $a_{i1}b_{i1} \in R$ but $b_{i1}t_i \notin R$ or such that $a_{i2}b_{i2} \in R$ but $s_{i+1}a_{i2} 
\notin R$, for some $i$; say the former. Then $R$ traverses the arc $b_{i1}b_{i2}$.  Since $c'(b_{i1}t_i)\ge c'(a_{i1}b_{i1})$, there is a path $L$ in ${\cal R}$ with $h(L)>0$ such 
that $b_{i1}t_i \in L$ but $a_{i1}b_{i1} \notin L$. Let $R_1$ (resp. $L_1$) be the subpath of $R$ (resp. $L$) from $\Lambda$ to $b_{i1}$ and let $R_2$ (resp. $L_2$)
be the subpath of $R$ (resp. $L$) from  $b_{i1}$ to $\Pi$. Note that $L_2$ consists of $b_{i1}t_i$ only. Set $P:=R_1 \cup L_2$ and $Q:=L_1 \cup R_2$.
Let $h'$ be obtained from $h$ by decreasing both $h(R)$ and $h(L)$ by $1$ and increasing both $h(P)$ and $h(Q)$ by $1$.
Let ${\cal R}'$ be obtained from ${\cal R}$ by adding $P$ and $Q$, if necessary. Replace $({\cal R}, h)$ by $({\cal R}', h')$ and repeat the process, 
we shall end up with a flow, denoted by $g$, in the path packing form as described in (5).  
 
\vskip 1mm
(6) We can find a $(\Lambda, \Pi)$-flow $g$ described in (5) in $O(n^4)$ time. Furthermore, the total number of paths $P$ with $g(P)>0$ is $O(n^2)$.

To justify this, let $N^*$ be obtained from $N$ by adding one arc parallel to $a_{i1}t_i$, denoted by $a_{i1}^*t_i^*$, and one arc 
parallel to $s_{i+1}b_{i2}$, denoted by $s_{i+1}^*b_{i2}^*$, for all $1\le i \le m-1$. Let $c^*$ be the following capacity function defined on
$N^*$:  

\vskip 1mm
$\bullet$ $c^*(a_{i1}^*t_i^*) := c'(a_{i1}b_{i1})=y_i^3$, 

\vskip 1mm

$\bullet$ $c^*(a_{i1}b_{i1}):=0$,

\vskip 1mm
$\bullet$ $c^*(b_{i1}t_i):= c'(b_{i1}t_i) - c'(a_{i1}b_{i1}) = y_i^2$,

\vskip 1mm
$\bullet$ $c^*(s_{i+1}^*b_{i2}^*): = c'(a_{i2}b_{i2})=y_i^6$,

\vskip 1mm
$\bullet$ $c^*(a_{i2}b_{i2}):=0$, and
\vskip 1mm

$\bullet$ $c^*(s_{i+1}a_{i2}):= c'(s_{i+1}a_{i2})- c'(a_{i2}b_{i2}) = y_i^5$

\vskip 1mm

\noindent for all $1\le i \le m-1$, and $c^*(a)=c'(a)$ for any other arc $a$. By (5), the maximum value of a $(\Lambda, \Pi)$-flow in $(N^*, c^*)$
is equal to that of a $(\Lambda, \Pi)$-flow in $(N, c')$. We can find a $(\Lambda, \Pi)$-flow $h$ in the path packing form in $(N^*, c^*)$ in
$O(n^4)$ time. Furthermore, the total number of paths $P$ with $h(P)>0$ is $O(n^2)$ (see the proof of (5) for details). Clearly, $h$ can be 
transformed into a $(\Lambda, \Pi)$-flow $g$ as described in (5) and (6) in $O(n^3)$ time. This proves (6).

By Lemma \ref{lem:btw paths}, for each $P\in\mathcal{P}$ with $g(P)>0$, arcs in $\{\eta(a): a\in P \cap A'\}$ form a path $P'$ in $\mathcal{Q}$. 
Then $g$ induces an assignment $h: \mathcal{Q}\rightarrow \mathbb{Z}_+$ by setting $h(Q):=g(P)$ if $Q=P'$ for some $P\in\mathcal{P}$ with $g(P)>0$ 
and setting $h(Q):=0$ otherwise. Clearly, we have

\vskip 1mm
(7) $\sum_{Q\in\mathcal{Q}} h(Q)=\sum_{P\in\mathcal{P}} g(P) = \sum_{C\in\mathcal{C}} y(C)$. 

\vskip 1mm
(8) Each $Q\in\mathcal{Q}$ with $h(Q)>0$ is a lower segment of some cycle in $T$.

Assume the contrary. Then $Q$ must contain some $a_kb_k$ by Lemma \ref{lem:obs on cycles}(iv). Let $P$ be the path in $\mathcal{P}$ with $g(P)>0$, such 
that arcs in $\{\eta(a): a\in P \cap A'\}$ form $Q$. By (5), either $\{a_{i1}b_{i1}, b_{i1}t_i\} \subseteq P$ or $\{a_{i2}b_{i2}, s_{i+1}a_{i2}\} \subseteq P$. 
So either $b_kv_k$ or $v_{k+1}a_k$ is on $Q$, which contradicts the fact that $Q$ is a path. 

For $1\le i \le m-1$, let ${\cal Q}_i$ consist of all paths in $\mathcal{Q}$ that start from some join vertex $v_p$ and end at some $v_q$ with $ p\le i \le q-1$, and 
let $\mathcal{O}_i$ be the set of all cycles in $\mathcal{C}$ whose lower segments start from some join vertex $v_p$ and end at some $v_q$ with $ p\le i \le q-1$.
  
\vskip 1mm	
(9) $\sum_{Q\in\mathcal{Q}_i} h(Q)= \sum_{C\in\mathcal{O}_i} y(C)$. 

\vskip 1mm

To justify this, let $\mathcal{Q}_i^-$ be the set of all paths in $\mathcal{Q}$ that end at some $v_p$ with $p \le i$, and let 
$\mathcal{Q}_{i+1}^+$ be the set of all paths in $\mathcal{Q}$ that start from some $v_q$ with $q \ge i+1$. Let $\mathcal{O}_i^-$ be the set of 
all cycles in $\mathcal{C}$ whose lower segments end at some $v_p$ with $p \le i$, and let $\mathcal{O}_{i+1}^+$ be the set of all cycles in 
$\mathcal{C}$ whose lower segments start from some $v_q$ with $q \ge i+1$. Moreover, let $\mathcal{P}_i^-$ be the set of all paths in $\mathcal{P}$ 
that end at some $t_p$ with $p \le i$, and let $\mathcal{P}_{i+1}^+$ be the set of all paths in $\mathcal{P}$ that start from some $s_q$ with 
$q \ge i+1$. Then $\sum_{Q\in\mathcal{Q}_i^-} h(Q) = \sum_{P\in\mathcal{P}_i^-} g(P) = \sum_{C\in\mathcal{O}_i^-} y(C)$
by (2) and (4). Similarly, $\sum_{Q\in\mathcal{Q}_{i+1}^+} h(Q) = \sum_{C\in\mathcal{O}_{i+1}^+} y(C)$. It follows from (7) that 
$\sum_{Q\in\mathcal{Q}_i} h(Q)=\sum_{Q\in\mathcal{Q}} h(Q)-\sum_{Q\in\mathcal{Q}_i^-} h(Q) - \sum_{Q\in\mathcal{Q}_{i+1}^+} h(Q) =
\sum_{C\in\mathcal{C}} y(C) - \sum_{C\in\mathcal{O}_i^-} y(C) - \sum_{C\in\mathcal{O}_{i+1}^+} y(C) = \sum_{C\in\mathcal{O}_i} y(C)$, as desired.

\vskip 1mm
It remains to complete $h$ into a maximum cycle packing by adding corresponding upper segments. For this purpose, let $D$ be the digraph obtained from
$\cup_{i=1}^{m-1} B_i$ by deleting $v_iv_{i+1}$ for all $1\le i \le m-1$. Define a capacity function $\bar{c}$ on $D$ as follows: 

\vskip 1mm
$\bullet$ $\bar{c}(a_iv_i):=c(a_iv_i)-c'(a_{i1}t_i)=c(a_iv_i)-y_i^1$,
\vskip 1mm

$\bullet$ $\bar{c}(b_iv_i):=c(b_iv_i)-c'(b_{i1}t_i)= c(b_iv_i)-y_i^2-y_i^3$,

\vskip 1mm
$\bullet$ $\bar{c}(a_ib_i):=c(a_ib_i)-c'(a_{i1}b_{i1})- c'(a_{i2}b_{i2}) =c(a_ib_i)-y_i^3-y_i^6$,

\vskip 1mm
$\bullet$ $\bar{c}(v_{i+1}b_i):=c(v_{i+1}b_i)- c'(s_{i+1}b_{i2}) = c(v_{i+1}b_i)-y_i^4$, and

\vskip 1mm
$\bullet$ $\bar{c}(v_{i+1}a_i):=c(v_{i+1}a_i)- c'(s_{i+1}a_{i2}) =c(v_{i+1}a_i)-y_i^5-y_i^6$

\vskip 1mm
\noindent for $1\le i \le m-1$.
\vskip 1mm

(10) The maximum value of a $v_{i+1}$-$v_i$ flow in $(D, \bar{c})$ is equal to $\sum_{C\in\mathcal{O}_i} y(C)$ for $1\le i\le m-1$.

\vskip 1mm
To justify this, note that the upper segment of each cycle in $\mathcal{O}_i$ contains a $v_{i+1}$-$v_i$ path in $D$ by Lemma \ref{lem:obs on cycles}(iii). 
On the other hand, arcs in $B_i\de\{v_iv_{i+1}\}$ can be contained only in cycles in $(\cup_{j=1}^6\mathcal{C}_i^j)\cup \mathcal{O}_i$. 
Thus (10) holds by Lemma \ref{lem: saturated arcs} and the definition of $\bar{c}$.

\vskip 1mm
Combining (9) and (10), we obtain

\vskip 1mm
(11) $\sum_{Q\in\mathcal{Q}_i} h(Q)$ is equal to the maximum value of a $v_{i+1}$-$v_i$ flow in $(D, \bar{c})$ for $1\le i\le m-1$. 

\vskip 1mm

Let $\mathcal{Q}^*=\{Q\in \mathcal{Q}: h(Q)>0$ and the ends of $Q$  are distinct$\}$. If $\mathcal{Q}^* \ne \emptyset$, take $Q$ from it. Suppose $Q$ is from 
$v_p$ to $v_q$, with $p < q$.  For each $p\le k \le q-1$, find a directed path $R_k$ from $v_{k+1}$ to $v_k$ in $B_k \de \{v_kv_{k+1}\}$
such that $\bar{c}(a)>0$ for each arc $a$ on $R_k$. Let $R$ be the concatenation of all these $R_k$, and let $\delta:=\min \{h(Q), \, \bar{c}(a): a \in R\}$.
Let $C$ be a cycle in $Q \cup R$ and define $y(C):=\delta$. Replace $\bar{c}(a)$ by $\bar{c}(a)-\delta$ for each arc $a$ in $Q \cup R$, and replace
$\mathcal{Q}^*$ by $\mathcal{Q}^* \de \{Q\}$ if $h(Q)=\delta$. Repeat the process until $\mathcal{Q}^*$ becomes empty. The correctness of this algorithm
is guaranteed by (11). Since originally $\mathcal{Q}^*$ has size $O(n^2)$ and $D$ has $O(n)$ arcs, the total number of cycles $C$ with $y(C)>0$ 
produced is $O(n^3)$. \qed

\section{General Algorithms}

We have come up with efficient algorithms for finding minimum FAS's and maximum cycle packings in arc-weighted tournaments in ${\cal G}^*$. 
The purpose of this section is to design more general algorithms for tournaments in ${\cal G}$.

As proved by Chen et al. \cite{CDZZ2}, both $F_1$ and $G_1$ (see Figure 2) are CM. Let us consider algorithmic problems on these two tournaments.

\begin{lemma}\label{lem:F1 and G1}
Suppose each arc of $T\in \{F_1, G_1\}$ is associated with a nonnegative integral weight $w(e)$. Then both a minimum FAS and a maximum cycle packing in $(T, w)$
can be found in constant time. 
\end{lemma}
\vspace{-1mm}

{\bf Proof.} Let $\tau_w(T)$ denote the minimum weight of an FAS in $(T, w)$ and let $\nu_w(T)$ denote the maximum size of a cycle packing in $(T, w)$. 
Since $T$ is CM, $\tau_w(T)=\nu_w(T)$ for all $w$; this min-max relation will play an important role in our algorithm design. 

Let $C_1, C_2, \ldots, C_p$ be all cycles in $T$, and let $K_1, K_2, \ldots, K_q$ be all inclusionwise minimal feedback arc sets in $T$ (see
\cite{CDZZ2} for details). Compute $w(K_i)$ for all $1\le i \le q$; we can thus obtain a minimum FAS in $(T, w)$ in constant time.  

To find a maximum cycle packing $y$ in $(T, w)$, we proceed by determining $y(C_1), y(C_2), \ldots, \\y(C_p)$ in order. 
Set $w_1:=w$. Suppose we have already obtained $y(C_1), y(C_2), \ldots, y(C_{i-1})$ and a weight function $w_i$ 
for some $i$ with $1\le i \le p$. Compute $\tau_{w_i}(T)$ using the above algorithm for the FAS problem. 
If $\tau_{w_i}(T)=0$, set $y(C_j):=0$ for all $j \ge i$, stop. Otherwise, find $\beta_{i,j}:=|C_i\cap K_j|$ for 
all $1\le j\le q$. Let $\alpha$ be an arbitrary integer between $0$ and $\min\{w_i(a): a\in C_i\}$, and let $c({\alpha})$ be the weight function 
obtained from $w_i$ by replacing $w_i(e)$ with $w_i(e)-\alpha$ for each arc $e$ on $C_i$.  Observe that

\vskip 1mm
(1) $\nu_{c(\alpha)} (T)= \tau_{c(\alpha)} (T) = \min\{w_i(K_j)-\alpha\beta_{i,j}: 1\le j\le q\}$. 
\vskip 1mm

\noindent Furthermore, $y(C_i)$ can be set to $\alpha$ if and only if $\nu_{c(\alpha)} (T) = \nu_{w_i}(T) - \alpha = \tau_{w_i}(T)-\alpha$; subject
to this, $\alpha$ is maximized. Combining this with (1), we obtain

\vskip 1mm
(2) $\tau_{w_i}(T)-\alpha = \min\{w_i(K_j)-\alpha\beta_{i,j}: 1\le j\le q\}$.
\vskip 1mm

Let $f(\alpha)= \tau_{w_i}(T)-\alpha$, let $g(\alpha)=\min\{w_i(K_j)-\alpha\beta_{i,j}: 1\le j\le q\}$, and view both of
them as functions defined on ${\mathbb R}$. Clearly, $f(0)=g(0)$ and $g(\alpha)$ is piecewise linear. So we can find all 
intersection points of $f(\alpha)$ and $g(\alpha)$ (see (2)) in constant time. Let $\alpha^*$ be the largest integral 
intersection point between $0$ and $\min\{w_i(a): a\in C_i\}$ and let $y(C_i):= \alpha^*$. Set $w_{i+1}:=c(\alpha^*)$ and $i:=i+1$, 
repeat the process. Clearly, the algorithm returns a maximum cycle packing in $(T, w)$ and runs in constant time. \qed\\

We can finally establish the main result of this paper.

\vskip 1mm

{\bf Proof of Theorem \ref{thm:main}}. Let $T=(V,A)$ be a M\"{o}bius-free tournament with a nonnegative integral weight $w(e)$ on each arc $e$.
If $T\in \{F_1, G_1\}$ (see Figure 2), then both a minimum FAS and a maximum cycle packing in $(T, w)$ can be found in constant time
by Lemma \ref{lem:F1 and G1}. So we assume that $T\notin \{F_1, G_1\}$. By Theorem \ref{structure}, $T \in \cal{G}$, the class of 
all tournaments $T$ as depicted in Figure 3, where $m\ge1$, undirected/dotted edges in the figure can be directed arbitrarily, and all other 
arcs (that are not drawn) are directed from ``left'' to ``right''. 

By Lemma \ref{thm:notriangle packing}, the maximum cycle packing problem on $(T, w)$ can be reduced to that on some $(T', w')$ in $O(n^3)$ time, 
where $T' \in \cal{G}^*$, the class of all tournaments in $\cal{G}$, in which $A_i\de v_i$ is acyclic for all vertical blocks $A_i$.  
By Theorem \ref{CP}, a maximum cycle packing in $(T', w')$, and hence in $(T, w)$, can be found in $O(n^7)$ time. 

Since $T$ is CM, $\tau_w(T)=\nu_w(T)$ for all $w$. Hence both of them can be found in $O(n^7)$ time by our algorithm for the 
cycle packing problem. 

To find a minimum FAS in $(T, w)$, let $e_1, e_2, \ldots, e_q$ be all arcs in $T$, where $q= {n \choose 2}$. Set $w_1:=w$ and $K_0:=\emptyset$.
Find $\tau_{w_1}(T)$ using our algorithm for the cycle packing problem. Suppose we have scanned arcs $e_1, e_2, \ldots e_{i-1}$ and obtained
the weight function $w_i$, $K_{i-1}$, and $\tau_{w_i}(T)$ for some $i$ with $1\le i \le q$. Let $c_i$ be the weight function obtained from
$w_i$ by replacing $w_i(e_i)$ with $0$ and find $\tau_{c_i}(T)$ using our algorithm for the cycle packing problem. Set $K_i:= K_{i-1} \cup
\{e_i\}$ if $\tau_{w_i}(T)= \tau_{c_i}(T)+w_i(e_i)$ and $K_i:= K_{i-1}$ otherwise. (Thus $e_i \in K_i$ if $w_i(e_i)=0$.) Set $w_{i+1}:=c_i$ 
if $e_i \in K_i$ and $w_{i+1}:=w_i$ otherwise. Set $i:=i+1$, repeat the process until $i=q$; return $K_q$ (which is a minimum FAS in $(T, w)$).     

To prove the correctness of our algorithm, let us establish the following statements simultaneously.

(1) $w_i(e)=0$ if $e \in K_{i-1}$ and $w_i(e)=w(e)$ if $e \notin K_{i-1}$.

(2) No arc $e_t \notin K_i$ with $t\le i$ is contained in a minimum FAS in $(T, w_i)$.

(3) There is a minimum FAS in $(T, w_i)$ that contains $K_i$.

We apply induction on $i$. The statements (1)-(3) hold trivially for $i=1$. Suppose they have been established for $i$. Let
us proceed to the step for $i+1$.

According to the algorithm,  $w_{i+1}$ is obtained from $w_i$ by replacing $w_i(e_i)$ with zero if $e_i \in K_i$ and $w_{i+1}=w_i$
otherwise. It follows instantly from the induction hypothesis that $w_{i+1}(e)=0$ if $e \in K_i$ and $w_{i+1}(w)=w(e)$ if $e \notin K_i$.
So (1) holds. 

If $e_{i+1} \notin K_{i+1}$, then $\tau_{w_{i+1}}(T)\ne \tau_{c_{i+1}}(T)+w_{i+1}(e_{i+1})$ by the algorithm. So $e_{i+1}$ is not contained 
in any minimum FAS in $(T, w_{i+1})$. Assume on the contrary that some arc $e_t \notin K_{i+1}$ with $t\le i$ is contained in a minimum FAS $J$ 
in $(T, w_{i+1})$. Observe that if $e_i\in K_i$, then $w_i(J \cup K_i)=w_{i+1}(J)+w_i(e_i)=\tau_{w_{i+1}}(T)+w_i(e_i)= \tau_{w_i}(T)$, where 
the first equality follows from (1). If $e_i \notin K_i$, then $w_i(J \cup K_i) = w_i(J \cup K_{i-1}) = w_{i+1}(J) =\tau_{w_{i+1}}(T) = \tau_{w_i}(T)$. 
In either case, $J \cup K_i$ would be a minimum FAS in $(T, w_i)$ that contains $e_t$, with $t \le i$, outside $K_{i+1}$ (and hence outside $K_i$);
this contradiction to induction hypothesis proves (2). 

If $e_{i+1} \in K_{i+1}$, then $\tau_{w_{i+1}}(T)= \tau_{c_{i+1}}(T)+w_{i+1}(e_{i+1})$ by the algorithm. So $e_{i+1}$ is contained in a 
minimum FAS in $(T, w_{i+1})$. By (1),  $w_{i+1}(e)=0$ if $e \in K_i$. It follows that there is a minimum FAS in $(T, w_{i+1})$ that contains 
$K_{i+1}$. Hence (3) is true.

By (2) and (3), $K_q$ is a minimum FAS in $(T, w_q)$. So $\tau_{w_q}(T)=w_q(K_q)$. Suppose we have proved that $\tau_{w_{i+1}}(T)=w_{i+1}(K_q)$
for some $i \le q-1$. If $w_{i+1}=c_i$, then $\tau_{w_i}(T)= \tau_{c_i}(T)+w_i(e_i)$ and $e_i \in K_i \subseteq K_q$, so $\tau_{w_{i}}(T)=
w_{i+1}(K_q)+ w_i(e_i)=w_i(K_q)$. If $w_{i+1}=w_i$, then $\tau_{w_{i}}(T)=\tau_{w_{i+1}}(T)=w_{i+1}(K_q)=w_i(K_q)$. In either case, 
$\tau_{w_i}(T)=w_q(K_i)$ for $q\ge i \ge 1$. Therefore $K_q$ is a minimum FAS in $(T, w_1)$ and hence in $(T, w)$. 

Clearly, our algorithm for the feedback arc set problem on $(T, w)$ runs in $O(n^9)$ time. \qed

\end{document}